\documentclass[a4paper, 11pt, reqno]{amsart}
\usepackage[margin=2.5cm]{geometry}
\usepackage{graphicx}
\usepackage{amsmath}
\usepackage{amsthm,amsfonts,amssymb,mathrsfs,amscd,amstext,amsbsy}
\usepackage{epic,eepic}
\usepackage{enumerate}
\usepackage[all]{xy}
\usepackage{tikz}
\usepackage{color}
\usepackage{verbatim}

\usetikzlibrary{patterns}
\usetikzlibrary{decorations.markings}
\usetikzlibrary{arrows}

\def\COMMENT#1{}

\newtheorem{theorem}{Theorem}[section]

\newtheorem{lem}[theorem]{Lemma}
\newtheorem{conj}[theorem]{Conjecture}
\newtheorem{que}[theorem]{Question}
\newtheorem{claim}[theorem]{Claim}

\newtheorem{fact}{Fact}

\numberwithin{theorem}{section}
\numberwithin{equation}{section}

\def\qed{\hfill \ifhmode\unskip\nobreak\fi\quad\ifmmode\Box\else$\Box$\fi\\ }

\tikzset{middlearrow/.style={
        decoration={markings,
            mark= at position 0.6 with {\arrow{#1}} ,
        },
        postaction={decorate}
    }
}

\tikzset{middleuparrow/.style={
        decoration={markings,
            mark= at position 0.75 with {\arrow{#1}} ,
        },
        postaction={decorate}
    }
}

\tikzset{middleupuparrow/.style={
        decoration={markings,
            mark= at position 0.9 with {\arrow{#1}} ,
        },
        postaction={decorate}
    }
}
\title{Sparse spanning $k$-connected subgraphs in tournaments}

\author{Dong Yeap Kang, Jaehoon Kim, Younjin Kim, and Geewon Suh}
\address[Dong Yeap Kang]{Department of Mathematical Sciences, KAIST, 291 Daehak-ro Yuseong-gu Daejeon, 305-701 South Korea}
\address[Jaehoon Kim]{School of Mathematics, University of Birmingham, Edgbaston, Birmingham B15 2TT, United Kingdom}
\address[Younjin Kim]{Institute of Mathematical Sciences, Ewha Womans University, Seoul, South Korea}
\address[Geewon Suh]{School of Electrical Engineering, KAIST, 291 Daehak-ro Yuseong-gu Daejeon, 305-701 South Korea}
\email{dynamical@kaist.ac.kr}
\email{kimjs@bham.ac.uk, mutualteon@gmail.com}
\email{younjinkim@ewha.ac.kr}
\email{gwsuh91@kaist.ac.kr}

\thanks{The first author was supported by the National Research Foundation of Korea (NRF) grant funded by the Korea government (MSIP) (No. NRF-2017R1A2B4005020) and also by TJ Park Science Fellowship (D.~Kang). 
The second author was also partially supported by the European Research Council under the European Union's Seventh Framework Programme (FP/2007--2013) / ERC Grant Agreements no. 306349 (J.~Kim). 
The third author was supported  by Basic Science Research Program through the National Research Foundation of Korea(NRF) funded by the Ministry of Education (No. 2017R1A6A3A04005963) (Y.~Kim). 
The fourth author was supported by Basic Science Research Program through the National Research Foundation of Korea(NRF) funded by the Ministry of Science, ICT \& Future Planning (2011-0011653) (G.~Suh). }

\date{\today}

\begin{document}
\begin{abstract}
In 2009, Bang-Jensen asked whether there exists a function $g(k)$ such that every strongly $k$-connected $n$-vertex tournament contains a strongly $k$-connected spanning subgraph with at most $kn + g(k)$ arcs. In this paper, we answer the question by showing that every strongly $k$-connected $n$-vertex tournament contains a strongly $k$-connected spanning subgraph with at most $kn + 750k^2\log_2(k+1)$ arcs, and there is a polynomial-time algorithm to find the spanning subgraph.
 \end{abstract}
 
\maketitle

\section{Introduction}
Search of certain subgraphs which inherit the properties of the original graph has a long history.  For example, Hajnal \cite{Haj} and Thomassen \cite{Thom} proved that a graph $G$ with high enough connectivity has two vertex disjoint $k$-connected subgraphs which together cover all vertices. Thomassen \cite{CTbip} also made a conjecture that a graph $G$ with high enough connectivity has a $k$-connected spanning bipartite subgraph.

For directed graphs, such problems become more difficult. One of most important problems in this direction is the following {\em MSSS$_k$ problem}, where MSSS$_k$ stands for Minimum Spanning Strongly $k$-connected Subgraph: for a given strongly $k$-connected digraph $D$, find a spanning strongly $k$-connected subgraph of $D$ with as few arcs as possible. For $k=1$, we call it {\em MSSS problem} by omitting $k$. It is known that the Hamilton cycle problem can be solved if one can solve the MSSS problem. Thus MSSS problem is a generalization of Hamilton cycle problem, so it has been studied extensively (see e.g \cite{bang-jensen2, BG Book} for a survey). Since the Hamilton cycle problem is NP-hard for general directed graphs, MSSS problem is also NP-hard for general directed graphs. Thus it makes sense to consider subclasses of directed graphs for this problem, and this problem is solvable in polynomial-time for several classes of graphs (see \cite{BGY, BJY}). In particular, MSSS problem for tournaments is trivial as any strongly-connected tournament contains a Hamilton cycle (see~\cite[Corollary 1.5.2]{BG Book}). However, it is not known whether MSSS$_k$ problem is solvable in polynomial-time for tournaments for $k\geq 2$. 

Naturally, one can ask about the size (the number of arcs) of minimum spanning strongly $k$-connected subgraphs for strongly $k$-connected tournaments. If we consider the same question for arc-connectivity, the following theorem was proved by Bang-Jensen, Huang and Yeo in 2004.

\begin{theorem}\cite{bang-jensen}\label{arc strong}
For $k \geq 1$, every strongly $k$-arc-connected $n$-vertex tournament contains a strongly $k$-arc-connected spanning subgraph $D$ with $|E(D)| \leq nk + 136k^2$.
\end{theorem}
This gives us an upper bound of the number of arcs in minimum spanning strongly $k$-arc-connected subgraphs for strongly $k$-arc-connected tournaments. However, for vertex-connectivity, no good upper bound was known. Indeed, Bang-Jensen~\cite{bang-jensen2} asked the following question in 2009.
\begin{que}\label{question1}\cite{bang-jensen2}
For $k \geq 1$, does there exist a function $g=g(k)$ such that every strongly $k$-connected $n$-vertex tournament has a strongly $k$-connected spanning subgraph with at most $kn + g(k)$ arcs?
\end{que}
In this paper, we answer this question by proving the following theorem.  

\begin{theorem}\label{main theorem}
For $k \geq 1$, every strongly $k$-connected tournament $T$ with $n$ vertices has a strongly $k$-connected spanning subgraph $D$ with at most $kn + 750 k^2 \log_2 (k+1)$ arcs.
\end{theorem}
Thus $g(k) = 750 k^2 \log_2 (k+1)$ is sufficient for answering Question~\ref{question1}, and this is asymptotically best possible up to logarithmic factor. Indeed, Bang-Jensen, Huang and Yeo~\cite{bang-jensen} introduced an $n$-vertex tournament $\mathcal{T}_{n,k}$ for $n\geq k$ such that every strongly $k$-arc-connected spanning subgraph of $\mathcal{T}_{n,k}$ contains at least $nk + \frac{k(k-1)}{2}$ arcs. Since every strongly $k$-connected digraphs are also strongly $k$-arc-connected, this example shows that Theorem~\ref{main theorem} is asymptotically best possible up to logarithmic factor. We conjecture that we can reduce $g(k)$ to $O(k^2)$.

\begin{conj}
There is $C>0$ such that for any positive integer $k$, every strongly $k$-connected $n$-vertex tournament $T$ contains a strongly $k$-connected spanning subgraph $D$ with at most $kn + Ck^2$ arcs.
\end{conj}

One of two main ingredients for the proof of Theorem~\ref{main theorem} is Lemma~\ref{main lemma} which is, roughly speaking, a tool guaranteeing a sparse linkage structure from/to certain vertex-sets for any tournament. The other main ingredient is ``robust linkage structures" introduced by K\"uhn, Lapinskas, Osthus and Patel in \cite{KLOP} to prove a conjecture of Thomassen on edge-disjoint Hamilton cycles in highly connected tournaments. Robust linkage structure is a very powerful tool for studying highly connected tournament. Further results were obtained by this method \cite{KKO, KOT, P1, P2}. The novelty of the proof of Theorem~\ref{main theorem} is that it produces a highly connected `sparse' subgraph in the tournament, whereas previous applications of the method only produced
highly connected relatively dense subgraphs.

\section{Basic terminology and tools}

For any positive integer $N \geq 1$, $[N]$ denotes the set $\left \{1,\dots , N \right \}$.  Let $\log:= \log_2$, where we omit the base 2. A {\em graph} or {\em simple graph} is an undirected graph without multiple edges between two vertices and loops. A {\em directed graph} or {\em digraph} $D = (V,E)$ is a pair of a vertex set $V(D)=V$ and an arc set $E(D)=E$, where $E$ is a collection of ordered pairs in $V\times V$. We let $\overrightarrow{uv}$ denote $(u,v) \in V\times V$ an {\em arc from $u$ to $v$}. An {\em oriented graph} is a digraph obtained by orienting each edge $e \in E(G)$ for a simple graph $G$. An $n$-vertex {\em tournament} is an oriented graph obtained by orienting each edge $e \in E(K_n)$, where $K_n$ is a simple complete graph of order $n$. For a set $S$ of vertices, $D-S$ denotes the induced digraph $D[V(D)\setminus S]$. For a set $E'$ of arcs, $D-E'$ denotes the digraph $(V(D), E(D)\setminus E')$. We say a digraph $D'=(V',E')$ is a {\em subgraph of $D=(V,E)$} if $V'\subseteq V$ and $E'\subseteq E$. We denote $D'\subseteq D$ if $D'$ is a subgraph of $D$.

For a collection of arcs $E$, we let $V(E):=\{u : \exists ~v \text{ such that } \overrightarrow{uv} \in E \text{ or } \overrightarrow{vu} \in E\}.$ A {\em path} always denotes a directed path. A path $P=(v_1,v_2,\dots,v_n)$ is called a {\em path from $v_1$ to $v_n$}, and we say $v_i$ is the {\em $i$th vertex} of $P$. Sometimes, we consider the path $P$ as a collection of arcs and $V(P)$ denotes $\{v_1,\dots, v_n\}$. A directed graph $D = (V,E)$ is {\em strongly connected} if for any $u,v \in V$, there is a path from $u$ to $v$. We say that digraph $D$ is {\em strongly $k$-connected}, if $|V| \geq k+1$ and for $S \subseteq V$ with $|S| \leq k-1$, the digraph $D - S$ remains strongly connected. Similarly, $D$ is {\em strongly $k$-arc-connected}, if for $W \subseteq E$ with $|W| \leq k-1$, the digraph $D- W$ remains strongly connected. It is easy to see that every strongly $k$-connected digraph is strongly $k$-arc-connected. For a directed graph $D = (V,E)$ and $v \in V$, let 
$$N^+_D (v) := \{ u\in V(D) : \overrightarrow{vu} \in E(D)\} \text{ and }N^-_D(v):= \{u\in V(D): \overrightarrow{uv} \in E(D)\}.$$ We call $u$ an {\em out-neighbor} of $v$ if $\overrightarrow{vu} \in E(D)$ and $u$ an {\em in-neighbor} of $v$ if $\overrightarrow{uv} \in E(D)$.
We define 
\begin{align*}
&d^{+}_D (v) := |N^+_D (v)|,\enspace d^{-}_D (v) := |N^{-}_D (v)|,\enspace d_D(v):= d^+_D(v)+ d^-_D(v),\\
&\delta^{+}(D) = \min_{v \in V(D)}{d^+_{D} (v)}, \enspace \delta^{-}(D) = \min_{v \in V(D)}{d^-_{D} (v)} ~\text{ and } ~ \delta (D) = \min_{v \in V(D)}{d_{D} (v)}.
\end{align*}
For a digraph $D$, $B \subseteq V(D)$ {\em out/in-dominates} $C \subseteq V(D)$ if every vertex in $C$ is an out/in-neighbor of a vertex in $B$, respectively. A tournament $T$ is {\it transitive} if $V(T)$ can be ordered into $v_1,\dots, v_n$ such that $\overrightarrow{v_iv_j}\in E(T)$ if and only if $i<j$. 
We say that $T$ is a transitive tournament {\em with respect to} the ordering $\sigma = (v_1,\dots, v_n)$ with the {\em source vertex} $v_1$ and the {\em sink vertex} $v_n$.

We say a directed path $P=(v_1,\dots,v_p)$ in $T$ is {\em backwards-transitive} if $\overrightarrow{v_iv_j} \in E(T)$ whenever $i\geq j+2$. For a vertex $v$ and a vertex-set $U = \left \{u_1 , \dots , u_k \right \}$, a collection $\{P_1,\dots, P_k\}$ of $k$ paths is a {\em $k$-fan from $v$ to $U$} if $P_i$ is a path from $v$ to $u_i\in U$, $U\cap V(P_i)=\{u_i\}$ for each $i\in [k]$, and $V(P_i)\cap V(P_j)=\{v\}$ for distinct $i,j\in [k]$. Similarly, a collection $\{P_1,\dots, P_k\}$ of $k$ paths is a {\em $k$-fan from $U$ to $v$} if $P_i$ is a path from $u_i\in U$ to $v$, $U\cap V(P_i)=\{u_i\}$ for each $i\in [k]$, and $V(P_i)\cap V(P_j)=\{v\}$ for distinct $i,j\in [k]$.

We will use the following well-known fact deduced from Menger's theorem later. We omit the proof.

\begin{fact}\label{Menger}
For any strongly $k$-connected digraph $D$, a vertex $v \in V(D)$ and $U\subseteq V(D)$ with $|U|\geq k$, there exists a $k$-fan from $v$ to $U$ and a $k$-fan from $U$ to $v$.
\end{fact}

Note that if $v\in U$, then one of the paths in the $k$-fan is a trivial path from $v$ to $v$.

\begin{lem}\label{degree}
For positive integers $n, k$ with $n \geq 2$ and $k \leq n$, an $n$-vertex tournament $T$ has at least $k$ vertices of out-degree at least $(n-k)/2$ and $k$ vertices of in-degree at least $(n-k)/2$. Moreover, $T$ has a vertex $v$ with $n/4 \leq d^{+}_T(v)\leq 3n/4$ and a vertex $u$ with $n/4 \leq d^{-}_T(u)\leq 3n/4$.
\end{lem}

\begin{proof}
Note that any $n$-vertex tournament contains a vertex with out-degree at least $(n-1)/2$. Let $v_1,\dots,v_n$ be an ordering of $V(T)$ such that $d^+_T(v_1)\geq \dots \geq d^{+}_T(v_n)$. Then $T[\{v_k,\dots, v_n\}]$ contains a vertex with out-degree at least $(n-k)/2$, thus $d^{+}_T(v_k)\geq (n-k)/2$. Hence $T$ contains $k$ vertices of out-degree at least $(n-k)/2$. It follows that $T$ also contains $k$ vertices of in-degree at least $(n-k)/2$ by reversing every arc of $T$ and applying the same argument.

This also gives us at least $\lfloor n/2 \rfloor$ vertices with out-degree at least $\frac{n - \lfloor n/2 \rfloor}{2} \geq n/4$, and at least $\lceil n/2 \rceil + 1$ vertices with in-degree at least $\frac{n - \lceil n/2 \rceil- 1}{2} \geq \frac{n}{4}-1$. Hence there exists a vertex $v$ with $n/4 \leq d^{+}_T(v) \leq (n-1)-(n/4-1) = 3n/4$. By reversing every arc of $T$ and applying the same argument, it follows that there is a vertex $u$ with $n/4 \leq d^{-}_T(u)\leq 3n/4$.
\end{proof}

We introduce the following useful lemmas regarding in-dominating sets and out-dominating sets of tournaments.

\begin{lem}\label{in-dominating sets}
Let $v$ be a vertex in an $n$-vertex tournament $T$ with $d^{+}_T(v)=d$. Then there exist $A\subseteq V(T)$ and a vertex $a \in A$ such that the following properties hold:
\begin{enumerate}
\item[$({\rm a}1)$] We have $\frac{1}{2}\log(d+1) + 1 \leq s \leq \frac{5}{2} \log(d+1)+2$ where $s=|A|$.
\item[$({\rm a}2)$] $T[A]$ is a transitive tournament with respect to the ordering $(v_1,\dots, v_{s})$ with source $v$ and sink $a$.
\item[$({\rm a}3)$] $A$ in-dominates $V(T)\setminus A$.
\item[$({\rm a}4)$] For $1 \leq i\leq s/5-13$, we have 
$$|N^+_T(v_i)\setminus A|, |N^-_T(v_i)\setminus A| \geq 8 d^{1/7} - 1.$$
\item[$({\rm a}5)$] For any positive integers $i, k$ with $1 \leq i\leq s - 5\log (k) - 30$, we have
$$|N^+_T(v_i)\setminus A|, |N^-_T(v_i)\setminus A| \geq 1000k^2.$$
\end{enumerate}
\end{lem}
\begin{proof}
Let $L_0 = V(T)$. If $d=0$, then let $L_1 = \emptyset$ and $A :=  \{v_1 \}$. Then it is obvious that $A$ with an ordering $(v_1)$ satisfies all (a1)--(a5). Now suppose $d \geq 1$. Let $v_1:=v$, $A_1:=\{v_1\}$ and $L_1 := N^+_T(v_1)$. Suppose $L_1 , \dots , L_i$ has already been defined with $|L_i| \geq 1$. If $L_i$ contains only one vertex $u$, let $v_{i+1}:=u$ and $A_{i+1} := A_i \cup \left \{v_{i+1} \right \}$. If $|L_i| \geq 2$, Lemma~\ref{degree} implies that there exists a vertex $u \in L_i$ with $|L_i|/4 \leq d^+_{T[L_i]} (u) \leq 3|L_i|/4$. Let $v_{i+1}:=u$ and $L_{i+1} := L_i \cap N^+_T (v_{i+1})$. This procedure gives vertices $v_1,\dots, v_s$ and sets $L_1,\dots, L_s$ with $L_s = \emptyset$. We let $A := A_s$ with ordering $(v_1,\dots, v_s)$ and let $a:=v_s$. From the construction, (a2) and (a3) are obvious.

%Note that (a3) easily follows from
%\begin{align}
%\bigcup_{i=1}^{s}{N_T^{-}(v_i)} \setminus \left \{v_1 , \dots , v_s \right \} = \bigcup_{i=1}^{s}{(L_{i-1} \setminus (L_i \cup \left \{v_i \right \}))} = V(T) \setminus \left \{v_1 , \dots , v_s \right \}
%\end{align}
 
The construction also implies that
\begin{align}\label{Ei Ei+1 size}
\frac{|L_i|}{4} \leq |L_{i+1}| \leq \frac{3|L_i|}{4} \text{ for } i\in[s-2] \enspace \text{ and }\enspace |L_{s-1}|=1.
\end{align}
Note that we have $s\geq 2$ because $d\geq 1$. This implies 
\begin{align}\label{Ei size}
(\frac{4}{3})^{s-i-1} \leq |L_i| \leq 4^{s-i-1} \text{ for }i\in [s-1].
\end{align}
In particular, \eqref{Ei size} with $i=1$ and the fact that $d=|L_1|$ together imply 
$$\frac{1}{2}\log (d) +2 \leq s \leq \frac{\log (d)}{2 - \log (3)} + 2 \leq \frac{5}{2}\log(d) + 2.$$ 
Thus we get (a1).

Note that $L_i \setminus (L_{i+1} \cup  \{ v_{i+1} \} ) \subseteq N^+_T(v_i)\setminus A$ and $L_{i-1}\setminus (L_i\cup \{v_i\}) \subseteq N^-_T(v_i)$. Thus, for $1 \leq i \leq s/5 -13$ we have
\begin{eqnarray*}
|N^+_T (v_i) \setminus A| &\geq& |L_i \setminus L_{i+1}| - 1 \stackrel{\eqref{Ei Ei+1 size}}{\geq} \frac{1}{4}|L_i| - 1  \stackrel{\eqref{Ei size}}{\geq} \frac{1}{4} (\frac{4}{3})^{s-i-1} - 1 \geq  \frac{1}{4} (\frac{4}{3})^{4s/5+12} - 1 \\
& \stackrel{({\rm a}1)}{\geq}  & \frac{1}{4}(\frac{4}{3})^{\frac{2}{5} \log (d+1) + 64/5} -1 \geq  8d^{1/7} - 1
\end{eqnarray*}
Similarly we also get $|N^-_T (v_i) \setminus A| \geq |L_{i-1} \setminus L_i | - 1 \geq 8d^{1/7} - 1.$ Thus (a4) holds.

For $i\leq s-5\log (k)-30$, \eqref{Ei size} implies that 
$$|L_i| \geq \left ( \frac{4}{3} \right )^{s - i - 1} \geq  \left ( \frac{4}{3} \right )^{5\log(k)+29} > 4100k^2.$$

Therefore, (a5) follows from
\begin{align*}
|N^+_T (v_i) \setminus A| &\geq |L_i \setminus L_{i+1}| - 1 \stackrel{\eqref{Ei Ei+1 size}}{\geq} \frac{1}{4}|L_i| - 1  \geq 1000 k^2, \enspace |N^+_T (v_i) \setminus A| \geq |L_{i-1}\setminus L_{i}|-1 \geq 1000 k^2.
\end{align*} 
\end{proof}

By reversing arcs of a tournament $T$ in Lemma~\ref{in-dominating sets}, we have the following analogue.

\begin{lem}\label{out-dominating sets}
Let $v$ be a vertex in an $n$-vertex tournament $T$ with $d=d^{-}_T(v)$. Then there exist $B\subseteq V(T)$ and a vertex $b \in B$ such that the following properties hold:
\begin{enumerate}
\item[$(\rm{b}1)$] We have $\frac{1}{2}\log(d+1) + 1 \leq s\leq \frac{5}{2} \log(d+1)+2$ where $s=|B|$
\item[$(\rm{b}2)$] $T[B]$ is a transitive tournament with respect to the ordering $(v_1,\dots, v_s)$ with source $b$ and sink $v$.
\item[$(\rm{b}3)$] $B$ out-dominates $V(T)\setminus B$.
\item[$(\rm{b}4)$] For $i\geq 4s/5+14$, we have
$$|N^+_T(v_i)\setminus B|, |N^-_T(v_i)\setminus B| \geq 8 d^{1/7} - 1.$$
\item[$(\rm{b}5)$] For any positive integers $i, k$ with $5\log(k)+31 \leq i\leq s$, we have 
$$|N^+_T(v_i)\setminus B|, |N^-_T(v_i)\setminus B| \geq 1000k^2.$$
\end{enumerate}
\end{lem}

\section{Sparse linkage structure}

In this section, we will prove Lemma~\ref{main lemma}. For an ordering $\sigma = (v_1,\cdots, v_n)$ of vertices, we say that an arc $\overrightarrow{v_iv_j}$ is {\em $\sigma$-forward} if $i<j$, and {\em $\sigma$-backward} if $j<i$. For two integers $a,b$, we let $\sigma(a,b):= \{ v_{\ell} : a\leq \ell \leq b, \ell \in [n]\}.$ For positive integers $n,k,t$, an $n$-vertex digraph $D$ and an ordering $\sigma$ of $V(D)$, we say an $D$ is {\em $(\sigma,k,t)$-good} if it satisfies the following.
\begin{enumerate}
\item[({\rm D}1)] Every arc in $D$ is a $\sigma$-forward arc.
\item[({\rm D}2)] Every vertex in $\sigma(1,n-t)$ has out-degree at least $k$ in $D$.
\item[({\rm D}3)] Every vertex in $\sigma(t+1,n)$ has in-degree at least $k$ in $D$. 
\end{enumerate}
Note that if $n\leq t$, then $\sigma(1,n-t)=\sigma(t+1,n)=\emptyset$, so ({\rm D}2) and ({\rm D}3) are vacuous. Also note that ({\rm D}2) or ({\rm D}3) never holds together with ({\rm D}1) if $t<k$. In Lemma~\ref{main lemma}, we will show that every almost complete oriented graph has a spanning subgraph $D'$ and an ordering $\sigma$ such that $D'$ is a sparse $(\sigma,k,t)$-good digraph for appropriate $k,t$. The following shows that $(\sigma,k,t)$-good digraph $D'$ provides a sparse linkage structure from/to certain vertex sets.

\begin{claim}\label{connectivity claim}
Let $k,t$ be two positive integers with $t\geq k$. Let $D'$ be a $(\sigma,k,t)$-good digraph for an ordering $\sigma$ of $V(D')$. Then for a set $S \subseteq V(D')$ of $k-1$ vertices and $v \in V(D')\setminus S$, there exists a path $P$ in $D'-S$ from $v$ to $\sigma(n-t+1,n)$ and a path $P'$ in $D'-S$ from $\sigma(1,t)$ to $v$. 
\end{claim}
\begin{proof} If $n \leq t$, then the claim is trivial as $\sigma(n-t+1,n)=\sigma(1,t)=V(D')$. Assume $n \geq t+1$. Let $\sigma=(v_1,\dots, v_n)$. Take a path $P$ starting at $v$ and ending at $v_j$ with the largest possible $j$. If $j\leq n-t$, then (D1) and (D2) imply that $v_j$ has at least $k$ out-neighbors with larger indices. Thus $N_{D'}^{+}(v_j)\setminus S$ contains a vertex $v_{j'}$ with $j'>j$. However, $P\cup \{\overrightarrow{v_{j}v_{j'}}\}$ contradicts the maximality of $j$. Thus we have $j > n-t$. Therefore there exists a path $P$ in $T-S$ from $v$ to $v_j \in \sigma(n-t+1,n)$. We can find $P'$ in a similar way.
\end{proof}

The following two claims are useful to prove Lemma~\ref{main lemma}. 

\begin{claim}\label{claim1}
For an integer $s \geq 0$, let $G$ be a bipartite graph with bipartition $A\cup B$ with $A=\{a_1,\dots, a_n\}, B=\{b_1\,\dots,b_n\}$ satisfying the following.
\begin{itemize}
\item[$(\text{P}1_s)$] For all $i,j\in [n]$ with $i<j$, we have $|N_G(a_i)\cap \{b_{i+1},\dots, b_{j}\}| \geq \frac{j-i-s}{2},$
\item[$(\text{P}2_s)$] for all $i,j\in [n]$ with $i<j$, we have $|N_G(b_j)\cap \{a_{i},\dots, a_{j-1}\}| \geq \frac{j-i-s}{2}.$
\end{itemize}
Then $G$ contains a matching of size at least $n-s-1$.
\end{claim}
\begin{proof} 
We may assume that $n-s-1 > 0$, otherwise the claim is obvious. By K\"{o}nig's theorem, it is enough to show that minimum vertex cover has size at least $n-s-1$.
Assume we have a minimum vertex cover $W$ of $G$. If $A\subseteq W$ or $B\subseteq W$, then $|W|\geq n \geq n-s-1$. So we may assume that each of $A\setminus W$ and $B\setminus W$ contains an element. Consider the smallest index $i$ such that $a_i \in A\setminus W$, and the largest index $j$ such that $b_{j} \in B\setminus W$. We have $i<j$, otherwise $W$ contains at least $n-1$ vertices. Then we have
$$ \{a_1,\dots, a_{i-1}\} \cup \{b_{j+1},\dots, b_{n}\} \cup (N_G (b_{j} ) \cap \{a_{i},\dots, a_{j - 1}\} ) \cup (N_G (a_i )\cap \{b_f{i+1},\dots,b_{j}\}) \subseteq W.$$
By $(\text{P}1_s)$ and $(\text{P}2_s)$, we have
$$|W| \geq i-1 + (n-j) + \frac{j-i-s}{2} + \frac{j-i-s}{2} \geq n-s-1$$
as desired. \end{proof}

\begin{claim}\label{claim2}
For $s\geq 0$, let $D$ be an $n$-vertex oriented graph with $\delta(D) \geq n-s-1$. Then there exists an ordering $\sigma = (v_1,\dots, v_n)$ of $V(D)$ that satisfies the following.
\begin{itemize}
\item[$(\text{Q}1_s)$] For any $i,j\in [n]$ with $i< j$, $v_i$ has at least $\frac{j-i-s}{2}$ out-neighbours in $\{v_{i+1},\dots,v_{j}\}$, 
\item[$(\text{Q}2_s)$] For any $i,j\in [n]$ with $i< j$, $v_j$ has at least $\frac{j-i-s}{2}$ in-neighbours in $\{v_{i},\dots, v_{j-1}\}$.
\end{itemize}
Moreover, we can find such an ordering in polynomial-time on $n$.
\end{claim}
\begin{proof}
%Note that such ordering exists because an ordering $\sigma$ which maximizes the number of $\sigma$-forward arcs in $D$ also satisfies both $(\text{Q}1_s)$ and $(\text{Q}2_s)$. Moreover, we can find a desired ordering in polynomial-time in the following way. 

We start with an arbitrary ordering $\sigma_1 = (v_1,\dots,v_n)$ of $V(D)$. Assume we have an ordering $\sigma_{\ell}$ of $V(D)$ for some $\ell \geq 1$. If $\sigma_{\ell}$ satisfies $(\text{Q}1_s)$ and $(\text{Q}2_s)$, then we are done. Otherwise consider $1 \leq i < j \leq n$ that does not satisfy $(\text{Q}1_s)$ or $(\text{Q}2_s)$. Let us define
$$ \sigma_{\ell+1}:= \left\{\begin{array}{ll}
(v_1,\dots, v_{i-1},v_{i+1},\dots,v_{j},v_{i},v_{j+1},\dots,v_n) & \text{if } i<j \text{ does not satisfy $(\text{Q}1_s)$,} \\
(v_1,\dots, v_{i-1},v_{j},v_{i},\dots,v_{j-1},v_{j+1},\dots,v_n) & \text{if } i<j \text{ does not satisfy $(\text{Q}2_s)$.}
\end{array}\right. $$ 
Note that $\sigma_{\ell+1}$ has at least one more $\sigma$-forward arc than $\sigma_{\ell}$. The number of $\sigma$-forward arcs in $D$ is at most $\binom{n}{2}$, so the procedure must end before we have $\sigma_{\binom{n}{2}}$. Thus we obtain a desired ordering in polynomial-time in $n$.
\end{proof}

Now we prove Lemma~\ref{main lemma}. It will be frequently used in the proof of Theorem~\ref{main theorem}. 

\begin{lem}\label{main lemma}
For integers $s \geq 0$ and $k \geq 1$, let $D$ be an $n$-vertex oriented graph with $\delta(D) \geq n-1-s$. Then there exist an ordering $\sigma$ of $V(D)$ and a $(\sigma,k,2k+s-1)$-good spanning subgraph $D'$ of $D$ with $|E(D')|\leq kn-k+sk$.
\end{lem}
\begin{proof}
If $n < 2k+s$, then an arbitrary ordering $\sigma$ of $V(D)$ with a digraph $D'$ with no arcs is $(\sigma,k,2k+s-1)$-good. Thus we may assume that $n \geq 2k+s$. By Claim~\ref{claim2}, we can find an ordering $\sigma=(v_1,\dots,v_n)$ which satisfies condition $(\text{Q}1_s)$ and $(\text{Q}2_s)$ in Claim~\ref{claim2}. We consider an auxiliary bipartite graph $H_0$ with a bipartition $A\cup B$,  where $A = \{ v_1,\dots, v_n \}$  and $B = \{ v'_1,\dots,v'_n \}$, such that $v_iv'_j \in H_0$ if and only if $\overrightarrow{v_iv_j}$ is a $\sigma$-forward arc of $D$. (i.e. $i<j$ and $\overrightarrow{v_iv_j} \in E(D)$.)

Note that the conditions $(\text{Q}1_s)$ and $(\text{Q}2_s)$ imply that 
the graph $H_0$ satisfies the condition $(\text{P}1_s)$ and $(\text{P}2_s)$. Assume we have a graph $H_{\ell}$ satisfying the condition $(\text{P}1_{s+2\ell})$ and $(\text{P}2_{s+2\ell})$.
By Claim~\ref{claim1}, $H_\ell$ contains a  matching $M_\ell$ of size at least $n-s-2\ell-1$.  Let $H_{\ell+1} := H_\ell \setminus M_\ell$. 
Then for any $i,j\in [n]$, we have $|N_{H_{\ell}}(a_i) \setminus N_{H_{\ell+1}}(a_i)| \leq 1$ and $|N_{H_{\ell}}(b_j)\setminus N_{H_{\ell+1}}(b_j)|\leq 1$. Thus the graph $H_{\ell+1}$ satisfies the condition $(\text{P}1_{s+2\ell+2})$ and $(\text{P}2_{s+2\ell+2})$.  
Repeating this for $0\leq \ell \leq k-1$ provides arc-disjoint matchings $M_0, M_1, \dots, M_{k-1}$ of $H_0$ where the size of $M_\ell$ is at least $n-s-2\ell-1$ for $0 \leq \ell \leq k-1$. By deleting some arcs, we may assume that for $0\leq \ell \leq k-1$ we have
\begin{align}\label{Mell size}
|E(M_\ell )| = n-s-2\ell-1.
\end{align}
Let $M$ be a subgraph of $H_0$ such that $E(M):= \bigcup_{\ell=0}^{k-1} E(M_\ell)$ and let $D_1$ be a subgraph of $D$ such that $$V(D_1):=V(D),\enspace E(D_1):= \{ \overrightarrow{v_iv_j} : v_iv'_j \in E(M)\}.$$
Then by construction of $H_0$, every arc of $D_1$ is a $\sigma$-forward arc and
\begin{align}\label{M size}
\Delta(M)\leq k \enspace \text{ and }\enspace |E(M)|= \sum_{\ell=0}^{k-1} |E(M_\ell)| \stackrel{\eqref{Mell size}}{=} kn - k^2 -sk.
\end{align}
Also this implies that 
\begin{align}\label{D1 size} 
&\Delta^{+}(D_1)\leq k, \enspace \Delta^{-}(D_1)\leq k, \enspace |E(D_1)|= kn-k^2 -sk, \nonumber \\
&d^{-}_{D_1}(v_i) \leq \min\{ k , i-1\} \enspace \text{ and } \enspace d^+_{D_1}(v_i) \leq \min\{k, n-i\}.
\end{align}
For each vertex $2k+s \leq i \leq n$, the number of $\sigma$-forward arcs towards $v_i$ in $D$ is at least $\lceil \frac{i-1-s}{2} \rceil \geq \lceil \frac{2k+s-1-s}{2} \rceil \geq k$ by ($\text{Q}2_{s}$). Thus for each $2k+s \leq i \leq n$, we can choose a set $N^-_i$ of $\sigma$-forward arcs towards $v_i$ such that $N^-_i\subseteq E(D)\setminus E(D_1)$ and $|N^-_i| = k-d^-_{D_1}(v_i)$. Similarly, for each $1\leq i\leq n-2k-s+1$, we can choose a set $N^+_i$ of $\sigma$-forward arcs from $v_i$ such that $N^+_i\cap E(D_1)=\emptyset$ and $|N^+_i| = k-d^+_{D_1}(v_i)$. 
Define a digraph $D' \subseteq D$ with 
$$V(D'):=V(D),\enspace E(D'):= E(D_1)\cup \bigcup_{i=2k+s}^{n} N^-_i \cup \bigcup_{i=1}^{n-2k-s+1} N^+_i.$$
Then $D'$ satisfies (D1) by construction, and satisfies (D2) since $|d^+_{D'}(v_i) |\geq d^+_{D_1}(v_i)+ |N^+_i| \geq k$ for $i \in [n-2k-s+1]$. Similarly, $D'$ also satisfies (D3), thus $D'$ is $(\sigma,k,2k+s-1)$-good.
Note that 
\begin{eqnarray*}
\left|\bigcup_{i=2k+s}^{n} N^-_i\right| &\leq&  \sum_{i=2k+s}^{n}(k-  d^-_{D_1}(v_i))  \enspace = \enspace k(n-2k-s+1) - \sum_{i=1}^{n} d^-_{D_1}(v_i) + \sum_{i=1}^{2k+s-1} d^-_{D_1}(v_i) \\
 &\stackrel{\eqref{D1 size}}{\leq} & k(n-2k-s+1) - |E(D_1)| + \sum_{i=1}^{2k+s-1} \min\{k,i-1\} \enspace \stackrel{\eqref{D1 size}}{=} \enspace \binom{k}{2} + sk.  
\end{eqnarray*}
Here, we get the second inequality because $E(D_1)= \sum_{i=1}^{n} d^-_{D_1}(v_i)$. Similarly, we also have $|\bigcup_{i=1}^{n-2k-s+1} N^+_i| \leq \binom{k}{2}+ sk.$
Thus we have
\begin{eqnarray*} 
|E(D')| &\leq& |E(D_1)| + \left|\bigcup_{i=2k+s}^{n} N^-_i\right| + \left|\bigcup_{i=1}^{n-2k-s+1}  N^+_i \right| \\
& \stackrel{\eqref{D1 size}}{\leq} & kn-k^2-sk + 2\binom{k}{2} + 2sk = kn -k + sk.
\end{eqnarray*}
\end {proof}

\section{Small tournaments}
In this section, we show that Theorem~\ref{main theorem} holds for any strongly $k$-connected tournament $T$ with at most $100 k\log(k+1)$ vertices. Note that Theorem~\ref{small tournament} is sufficient for our purpose. To prove Theorem~\ref{small tournament}, we use the following lemma, which is a modification of Lemma~2.1 in \cite{P1}, and the proof is almost identical except a few changes. 

\begin{lem}\cite{P1}\label{5m}
Let $k \geq 1$ and $n\geq 5k$ be integers.
Every $n$-vertex tournament $T$ contains two disjoint sets of vertices $X$ and $Y$ of size $k$ such that for any set $S$ of $k-1$ vertices and any $x\in X\setminus S, y\in Y\setminus S$ there is a path $P$ in $T-S$ from $x$ to $y$.
\end{lem}
\begin{proof}
Let $\overrightarrow{K_{k,k}}$ be a bipartite digraph with partition $A,B$ such that $|A|=|B|=k$ and for every $u\in A, v\in B$, we have $\overrightarrow{uv} \in E(\overrightarrow{K_{k,k}})$.  
If $T$ contains $\overrightarrow{K_{k,k}}$ with bipartition $A$ and $B$ as a subgraph, then $X:=A, Y:=B$ are sufficient for our purpose. Thus we may assume that $T$ does not contain $\overrightarrow{K_{k,k}}$ as a subgraph.

Let $X = \left \{x_1,\dots, x_k \right \}$ be a set of $k$ vertices in $T$ of largest out-degree and $\left \{ y_1,\dots,y_k \right \}$ be a set of $k$ vertices in $T$ of largest in-degree. Since $n \geq 5k$, we may assume $X \cap Y = \emptyset$. From Lemma~\ref{degree}, we have $d^+_T(x_i) \geq (n-k)/2\geq 2k$ and $d^-_T(y_i)\geq (n-k)/2\geq 2k$ for all $i\in [k]$. Consider a set $S\subseteq V(T)$ of size $k-1$.
For each $i,j\in [k]$ let $X_{i,j}:= N^+(x_i) \setminus N^-(y_j)$, $Y_{i,j}:= N^-(y_j)\setminus N^+(x_i)$, $I_{i,j} = N^+(x_i)\cap N^-(y_j)$. Let $M_{i,j}$ be a maximum matching between $X_{i,j}$ and $Y_{i,j}$ such that every arc is directed from $X_{i,j}$ to $Y_{i,j}$. 
For each $z\in I_{i,j}$, $T$ contains a path $(x_i,z,y_j)$ and for each $\overrightarrow{ww'} \in M_{i,j}$, $T$ contains a path $(x_i,w,w',y_j)$. Moreover, those paths are all pairwise internally vertex disjoint. Thus if $|M_{i,j}| + |I_{i,j}| \geq k$ for all $i,j\in [k]$, then for any $x_i$ and $y_j$, there are at least $k$ internally vertex disjoint paths from $x_i$ to $y_j$. So we are done since for each $i,j\in [k]$ at least one path from $x_i$ to $y_j$ does not intersect with $S$. If there exist $i,j \in [k]$ such that $|M_{i,j}| + |I_{i,j}| < k$, then we have
$$|X_{i,j} \setminus V(M_{i,j})|\geq | N^+_T(x_i) - I_{i,j} - V(M_{i,j})| \geq d^+_T(x_i) - k \geq k.$$ Similarly we get $|Y_{i,j}\setminus V(M_{i,j})|\geq k$. Since $M_{i,j}$ is a maximal matching from $X_{i,j}$ to $Y_{i,j}$, for any $x'\in X_{i,j}\setminus V(M_{i,j})$ and $y'\in Y_{i,j}\setminus V(M_{i,j})$ we have $\overrightarrow{y'x'} \in E(T)$. This contradicts the fact that $T$ does not contain $\overrightarrow{K_{k,k}}$.
\end{proof}

Now we prove the theorem, which has worse upper bound than the upper bound in Theorem~\ref{main theorem} for sufficiently large $n$. However, if $n$ is small enough, for example, $n \leq 100k \log(k+1)$, then the following theorem implies Theorem~\ref{main theorem}.

\begin{theorem}\label{small tournament}
For any integer $k\geq 1$, every strongly $k$-connected tournament $T$ contains a strongly $k$-connected spanning subgraph $D$ with $|E(D)|\leq  (5k-2)n + \binom{5k}{2}$.
\end{theorem}
\begin{proof}
If $T$ has less than $5k$ vertices, then $T$ itself is sufficient to be $D$. Otherwise, let $V'\subseteq V$ be a set of $5k$ vertices. By applying Lemma~\ref{5m}, we can find two disjoint sets $X=\{x_1,\dots, x_k\},Y=\{y_1,\dots,y_k\}$ of size $k$ such that for any set $S\subseteq V'$ of size $k-1$ and vertices $x\in X, y\in Y$, there exists a path from $x$ to $y$ in $T[V']-S$. 
We apply Lemma~\ref{main lemma} to $T$ with parameters $0,k$ corresponding to $s,k$, and we obtain an ordering $\sigma = (v_1,\dots, v_n)$ of $V(T)$ and a $(\sigma,k,2k-1)$-good spanning subgraph $D'\subseteq T$ with $|E(D')|\leq kn - k$. 

For each $n-2k+2\leq i\leq n$, let $\{ P(v_i,j) : j\in [k]\}$ be a $k$-fan from $v_i$ to $X$ (which exists since $T$ is strongly $k$-connected) such that $P(v_i , j)$ is a path from $v_i$ to $x_j$. Note that if $v_i=x_j$, then $P(v_i,j)$ is a path of one vertex.
Similarly, for each $1\leq i\leq 2k-1$, let $\{Q(v_i,j) : j\in [k]\}$  be a $k$-fan from $Y$ to $v_i$ such that $Q(v_i,j)$ is a path from $y_j$ to $v_i$.  Note that if $v_i=y_j$, then $Q(v_i,j)$ is a path of one vertex.

For each $n-2k+2\leq i\leq n$ and $1\leq i'\leq 2k-1$, it follows that
$$\sum_{j=1}^{k} |E(P(v_i,j))|\leq n-1,\enspace \sum_{j=1}^{k} |E(Q(v_{i'},j))|\leq n-1,$$
because no vertex other than $v_i$ is covered by two distinct paths in a $k$-fan from $v_i$ to $X$ or by two distinct paths in a $k$-fan from $Y$ to $v_i$. Let $D$ be the subgraph of $T$ such that 
$$V(D):=V(T), \enspace E(D):= E(T(V'))\cup E(D') \cup \bigcup_{i=1}^{2k-1}\bigcup_{j=1}^{k} Q(v_i,j) \cup \bigcup_{i=n-2k+2}^{n}\bigcup_{j=1}^{k} P(v_i,j).$$
Then 
\begin{align*} 
|E(D)|&\leq  |E(T(V'))| + |E(D')| + (2k-1)(n-1) + (2k-1)(n-1) \\
&\leq \binom{5k}{2} + kn - k + (4k-2)n \leq (5k-2)n + \binom{5k}{2}.
\end{align*}
Moreover, for any set $S\subseteq V(D)$ of $k-1$ vertices and any vertices $u,v\in V(T)\setminus S$, there is a path $P$ from $v$ to $v_i$ and a path $P'$ from $v_{i'}$ to $u$ in $D'-S$ for some $i\geq n-2k+2$ and $i'\leq 2k+1$, by Claim~\ref{connectivity claim}. Since $\{P(v_i,j): j\in [k]\}$ and $\{Q(v_{i'},j) : j\in [k] \}$ are $k$-fans, there are $s,s' \in [k]$ such that both $P(v_i,s)$ and $Q(v_{i'},s')$ do not intersect $S$. Let $x_s^* \in X$ and $y_{s'}^* \in Y$ be the endpoints of $P(v_i,s)$ and $Q(v_{i'},s')$, respectively. (note that if $v_i \in X$ ($v_{i'} \in Y$), then $x_s^* = v_i$ ($y_{s'}^* = v_{i'}$).)
By Claim~\ref{5m}, there is a path $P''$ in $T[V']-S$ from $x_s^*$ to $y_{s'}^*$. Hence $E(P)\cup E(P(v_i,s))\cup E(P'')\cup E(Q(v_{i'},s'))\cup E(P')$ contains a path in $D-S$ from $u$ to $v$. Thus $D$ is strongly $k$-connected. 
\end{proof}

\section{Proof of Theorem~\ref{main theorem}}

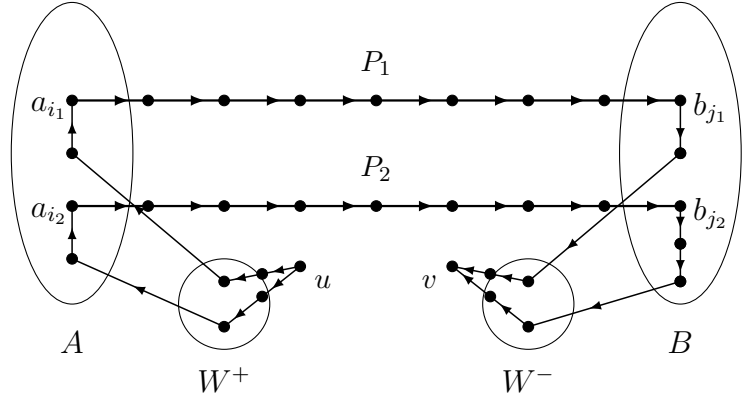
\begin{figure}
\centering
\begin{tikzpicture}[scale=1]

  \draw[fill=none] (8,0) ellipse [x radius=0.8,y radius=2];
  \node at (8,-2.5) {\large $B$};
 \draw[fill=none] (0,0) ellipse [x radius=0.8,y radius=2];
\node at (0,-2.5) {\large $A$};

 \draw[fill=none] (2,-2) ellipse [x radius=0.6,y radius=0.6];
\node at (2,-3) {\large $W^+$};

 \draw[fill=none] (6,-2) ellipse [x radius=0.6,y radius=0.6];
\node at (6,-3) {\large $W^-$};

\node at (-0.3,0.6) {\large $a_{i_1}$};
\node at (-0.3,-0.8) {\large $a_{i_2}$};

\node at (8.4,0.6) {\large $b_{j_1}$};
\node at (8.4,-0.8) {\large $b_{j_2}$};

\filldraw[fill=black] (0,0.7) circle (2pt);
\filldraw[fill=black] (1,0.7) circle (2pt);
\filldraw[fill=black] (2,0.7) circle (2pt);
\filldraw[fill=black] (3,0.7) circle (2pt);
\filldraw[fill=black] (4,0.7) circle (2pt);
\filldraw[fill=black] (5,0.7) circle (2pt);
\filldraw[fill=black] (6,0.7) circle (2pt);
\filldraw[fill=black] (7,0.7) circle (2pt);
\filldraw[fill=black] (8,0.7) circle (2pt);

\draw[middleuparrow={latex}, line width=0.3mm] (0,0.7) -- (1,0.7);
\draw[middleuparrow={latex}, line width=0.3mm] (1,0.7) -- (2,0.7);
\draw[middleuparrow={latex}, line width=0.3mm] (2,0.7) -- (3,0.7);
\draw[middleuparrow={latex}, line width=0.3mm] (3,0.7) -- (4,0.7);
\draw[middleuparrow={latex}, line width=0.3mm] (4,0.7) -- (5,0.7);
\draw[middleuparrow={latex}, line width=0.3mm] (5,0.7) -- (6,0.7);
\draw[middleuparrow={latex}, line width=0.3mm] (6,0.7) -- (7,0.7);
\draw[middleuparrow={latex}, line width=0.3mm] (7,0.7) -- (8,0.7);

\node at (4,1.2) {{\bf $P_1$}};

\filldraw[fill=black] (0,-0.7) circle (2pt);
\filldraw[fill=black] (1,-0.7) circle (2pt);
\filldraw[fill=black] (2,-0.7) circle (2pt);
\filldraw[fill=black] (3,-0.7) circle (2pt);
\filldraw[fill=black] (4,-0.7) circle (2pt);
\filldraw[fill=black] (5,-0.7) circle (2pt);
\filldraw[fill=black] (6,-0.7) circle (2pt);
\filldraw[fill=black] (7,-0.7) circle (2pt);
\filldraw[fill=black] (8,-0.7) circle (2pt);
\node at (4,-0.2) {{\bf $P_2$}};

\draw[middleuparrow={latex}, line width=0.3mm] (0,-0.7) -- (1,-0.7);
\draw[middleuparrow={latex}, line width=0.3mm] (1,-0.7) -- (2,-0.7);
\draw[middleuparrow={latex}, line width=0.3mm] (2,-0.7) -- (3,-0.7);
\draw[middleuparrow={latex}, line width=0.3mm] (3,-0.7) -- (4,-0.7);
\draw[middleuparrow={latex}, line width=0.3mm] (4,-0.7) -- (5,-0.7);
\draw[middleuparrow={latex}, line width=0.3mm] (5,-0.7) -- (6,-0.7);
\draw[middleuparrow={latex}, line width=0.3mm] (6,-0.7) -- (7,-0.7);
\draw[middleuparrow={latex}, line width=0.3mm] (7,-0.7) -- (8,-0.7);

\filldraw[fill=black] (3,-1.5) circle (2pt);
\filldraw[fill=black] (5,-1.5) circle (2pt);
\node at (3.3,-1.7) {{\bf $u$}};
\node at (4.7,-1.7) {{\bf $v$}};

\filldraw[fill=black] (2.5,-1.6) circle (2pt);
\filldraw[fill=black] (2.5,-1.9) circle (2pt);

\filldraw[fill=black] (5.5,-1.6) circle (2pt);
\filldraw[fill=black] (5.5,-1.9) circle (2pt);

\filldraw[fill=black] (2,-1.7) circle (2pt);
\filldraw[fill=black] (2,-2.3) circle (2pt);

\draw[middleuparrow={latex}, line width=0.2mm] (3,-1.5) -- (2.5,-1.6);
\draw[middleuparrow={latex}, line width=0.2mm] (2.5,-1.6) -- (2,-1.7);
\draw[middleuparrow={latex}, line width=0.2mm] (3,-1.5) -- (2.5,-1.9);
\draw[middleuparrow={latex}, line width=0.2mm] (2.5,-1.9) -- (2,-2.3);

%\filldraw[fill=black] (0.5,0.1) circle (2pt);
%\filldraw[fill=black] (1,-0.5) circle (2pt);
%\filldraw[fill=black] (1.5,-1.1) circle (2pt);

%\draw[middleuparrow={latex}, line width=0.2mm] (2,-1.7) -- (1.5,-1.1);
%\draw[middleuparrow={latex}, line width=0.2mm] (1.5,-1.1) -- (1,-0.5);
%\draw[middleuparrow={latex}, line width=0.2mm] (1,-0.5) -- (0.5,0.1);
%\draw[middleuparrow={latex}, line width=0.2mm] (0.5,0.1) -- (0,0.7);

\draw[middlearrow={latex}, line width=0.2mm] (2,-1.7) -- (0,0);
\draw[middlearrow={latex}, line width=0.2mm] (2,-2.3) -- (0,-1.4);

\draw[middlearrow={latex}, line width=0.2mm] (0,0) -- (0,0.7);
\draw[middlearrow={latex}, line width=0.2mm] (0,-1.4) -- (0,-0.7);

\filldraw[fill=black] (0,0) circle (2pt);
\filldraw[fill=black] (0,-1.4) circle (2pt);
\filldraw[fill=black] (8,0) circle (2pt);

\filldraw[fill=black] (8,-1.2) circle (2pt);
\filldraw[fill=black] (8,-1.7) circle (2pt);

%\filldraw[fill=black] (0.5,-1.1) circle (2pt);
%\filldraw[fill=black] (1,-1.5) circle (2pt);
%\filldraw[fill=black] (1.5,-1.9) circle (2pt);

%\draw[middleuparrow={latex}, line width=0.2mm] (2,-2.3) -- (1.5,-1.9);
%\draw[middleuparrow={latex}, line width=0.2mm] (1.5,-1.9) -- (1,-1.5);
%\draw[middleuparrow={latex}, line width=0.2mm] (1,-1.5) -- (0.5,-1.1);
%\draw[middleuparrow={latex}, line width=0.2mm] (0.5,-1.1) -- (0,-0.7);

\filldraw[fill=black] (6,-1.7) circle (2pt);
\filldraw[fill=black] (6,-2.3) circle (2pt);

\draw[middleuparrow={latex}, line width=0.2mm] (6,-1.7) -- (5.5,-1.6) ;
\draw[middleuparrow={latex}, line width=0.2mm] (6,-2.3) -- (5.5,-1.9);
\draw[middleuparrow={latex}, line width=0.2mm] (5.5,-1.6) -- (5,-1.5);
\draw[middleuparrow={latex}, line width=0.2mm] (5.5,-1.9) -- (5,-1.5);

%\filldraw[fill=black] (7.5,0.1) circle (2pt);
%\filldraw[fill=black] (7,-0.5) circle (2pt);
%\filldraw[fill=black] (6.5,-1.1) circle (2pt);

%\filldraw[fill=black] (7.5,-1.1) circle (2pt);
%\filldraw[fill=black] (7,-1.5) circle (2pt);
%\filldraw[fill=black] (6.5,-1.9) circle (2pt);

%\draw[middlearrow={latex}, line width=0.2mm] (8,0.7) -- (7.5,0.1) ;
%\draw[middlearrow={latex}, line width=0.2mm] (7.5,0.1)  -- (7,-0.5);
%\draw[middlearrow={latex}, line width=0.2mm] (7,-0.5) -- (6.5,-1.1);
%\draw[middlearrow={latex}, line width=0.2mm] (6.5,-1.1)-- (6,-1.7);

%\draw[middlearrow={latex}, line width=0.2mm] (8,-0.7) -- (7.5,-1.1);
%\draw[middlearrow={latex}, line width=0.2mm] (7.5,-1.1) -- (7,-1.5);
%\draw[middlearrow={latex}, line width=0.2mm] (7,-1.5) -- (6.5,-1.9);
%\draw[middlearrow={latex}, line width=0.2mm] (6.5,-1.9) -- (6,-2.3);

\draw[middleuparrow={latex}, line width=0.2mm] (8,0) -- (6,-1.7);
\draw[middlearrow={latex}, line width=0.2mm] (8,-1.7) -- (6,-2.3);

\draw[middleuparrow={latex}, line width=0.2mm] (8,0.7) -- (8,0);
\draw[middlearrow={latex}, line width=0.2mm] (8,-0.7) -- (8,-1.2);
\draw[middleuparrow={latex}, line width=0.2mm] (8,-1.2) -- (8,-1.7);

\filldraw[fill=black] (8,0) circle (2pt);

\filldraw[fill=black] (8,-1.2) circle (2pt);
\filldraw[fill=black] (8,-1.7) circle (2pt);

\end{tikzpicture}
\caption{ Two paths from $u$ to $v$ in the outline of the idea when $k=2$. }
\end{figure}

 \subsection*{Outline of the idea}
 For a strongly $k$-connected tournament $T$, we construct a set $A$ which is the union of many in-dominating sets, a set $B$ which is the union of many out-dominating sets and $k$ pairwise vertex disjoint paths $P_1,\dots, P_k$ from $A$ to $B$ such that the path $P_t$ is from $a_{i_t}$ to $b_{j_t}$ for each $t\in [k]$. We choose the size of in-dominating sets and out-dominating sets in $A$ and $B$ to be sufficiently small (Lemmas~\ref{in-dominating sets} and~\ref{out-dominating sets}) so that there are few vertices in both $A$ and $B$.
 
 To find a sparse subgraph $D$, we divide the vertex set $V(T)$ into $V_1,V'_1,V_2,V_3, V_4$ and apply Lemma~\ref{main lemma} to each set and find two small sets $W^+$ and $W^-$ such that $D$ contains $k$ internally vertex-disjoint paths from any vertex $u$ to $W^+$ and $k$ internally vertex-disjoint paths from $W^-$ to any vertex $v$. We also add some arcs to the subgraph $D$ so that there are $k$ arcs in $D$ from each vertex in $W^+$ to $A$, and $k$ arcs in $D$ from $B$ to each vertex in $W^-$. Note that this is possible since $A$ is a union of many in-dominating sets and $B$ is a union of many out-dominating sets. 
By adding some arcs inside $A$ and $B$, we can also ensure that there are $k$ internally vertex-disjoint paths from any vertex in $A$ to the vertices $a_{i_1}, \dots,a_{i_k}$ and $k$ internally vertex-disjoint paths from $b_{j_1},\dots b_{j_k}$ to any vertex in $B$.
Then for each distinct vertices $u,v \in V(T)$, the paths from $u$ to $W^+$, the arcs from $W^+$ to $A$, the paths inside $A$ to $a_{i_1},\dots, a_{i_k}$, the paths $P_1,\dots, P_k$, the paths inside $B$ from $b_{j_1},\dots,b_{j_k}$, the arcs from $B$ to $W^-$, and the paths from $W^-$ to $v$ all together form $k$ internally vertex-disjoint paths from $u$ to $v$ as in Figure 1. Since $u$ and $v$ are arbitrarily chosen, $D$ is strongly $k$-connected while $D$ is sparse enough. 

 \subsection*{Proof of Theorem~\ref{main theorem}}   
 Let $T$ be a strongly $k$-connected $n$-vertex tournament with a vertex-set $V$. Note that Theorem~\ref{main theorem} is trivial for $k=1$ since every strongly connected $n$-vertex tournament contains a Hamilton cycle (see~\cite[Theorem 1.5.1]{BG Book}). There is  an algorithm that finds a Hamilton cycle in an $n$-vertex tournament and runs in $O(n^2)$ (see~\cite{yannis}). If $k \geq 2$ and $n\leq 100 k\log(k+1)$, then Theorem~\ref{small tournament} implies Theorem~\ref{main theorem}. Thus we may assume that
 \begin{align*}
 k\geq 2, \enspace n > 100k\log(k+1).
 \end{align*} 
Now we construct an appropriate in-dominating set $A$ and out-dominating set $B$ as we sketched before. Let $X$ and $Y$ be two disjoint sets such that $X$ is a set of $3k-1$ vertices with smallest out-degrees, and let $Y$ is a set of $3k-1$ vertices with smallest in-degrees. Let $\delta^- := \max_{y\in Y} d^{-}_T(y)$ and $ \delta^+ := \max_{x \in X} d^{+}_T(x)$. Without loss of generality, we assume 
\begin{align}\label{delta size}
\delta^- \geq \delta^+.
\end{align}
Choose $x_1\in X$ having the largest number of out-neighbors in $V\setminus (X\cup Y)$ among all vertices in $X$, and let 
$$d^+_1 := |(V\setminus (X\cup Y))\cap N^+_{T}(x_1)|.$$
We apply Lemma~\ref{in-dominating sets} with $T-((X-\left\{x_1 \right \}) \cup Y), x_1, d^+_1$ corresponding to $T,v,d$ to find a set $A_1$ and a sink vertex $a_1\in A_1$ satisfying (a1)--(a5). Note that (a1) implies that $A_1$ is nonempty and $a_1=x_1$ could happen when $d^+_1=0$.
For given $x_1,\dots, x_{i}$ and $A_1,\dots, A_i$, let us choose $x_{i+1}\in X\setminus\{x_1,\dots, x_{i}\}$ having the largest number of out-neighbours in $V\setminus (X\cup Y\cup \bigcup_{j=1}^{i} A_j)$ among all the vertices in $X\setminus\{x_1,\dots,x_{i}\}$ and let $$d^+_{i+1} := |(V\setminus (X\cup Y \cup \bigcup_{j=1}^{i} A_j))\cap N^+_{T}(x_{i+1})|.$$
We apply Lemma~\ref{in-dominating sets} with $T-((X - \left \{ x_{i+1} \right \}) \cup Y\cup \bigcup_{j=1}^{i} A_j), x_{i+1}, d^+_{i+1}$ corresponding to $T,v,d$ to find a set $A_{i+1}$ and a sink vertex $a_{i+1}\in A_{i+1}$ satisfying (a1)--(a5). By repeating this $3k-1$ times, we get $A_{1},\dots,A_{3k-1}$ and $a_1,\dots, a_{3k-1}$. We let $A:=\bigcup_{i=1}^{3k-1} A_i.$ 
 
Next, we choose  $y_1\in Y$ having the largest number of in-neighbours in $V\setminus (X\cup Y\cup A)$. Let 
$$d^{-}_1 := |(V\setminus(X\cup Y \cup A))\cap N^-_{T}(y_1)|.$$ 
Then we apply Lemma~\ref{out-dominating sets} with $T-(X\cup (Y-\left \{y_1 \right \})\cup A), y_1, d^-_1$ corresponding to $T,v,d$ to find a set $B_1$ and a source vertex $b_1\in B_1$ satisfying (b1)--(b5). Note that (b1) implies that $B_1$ is nonempty and $b_1=y_1$ could happen when $d^-_1=0$.
For given $A, y_1,\dots,y_{i}$ and $B_1,\dots, B_i$, let us choose $y_{i+1}\in Y\setminus\{y_1,\dots, y_{i}\}$ having the largest number of in-neighbours in $V\setminus (X\cup Y\cup A\cup \bigcup_{j=1}^{i} B_j)$ among all the vertices in $Y\setminus\{y_1,\dots,y_{i}\}$ and let $$d^-_{i+1} := |(V\setminus (X\cup Y \cup A\cup \bigcup_{j=1}^{i} B_j))\cap N^-_{T}(y_{i+1})|.$$
We apply Lemma~\ref{out-dominating sets} with $T-(X\cup (Y - \left \{ y_{i+1} \right \})\cup A\cup \bigcup_{j=1}^{i} B_j), y_{i+1}, d^-_{i+1}$ corresponding to $T,v,d$ to find a set $B_{i+1}$ and a source vertex $b_{i+1}\in B_{i+1}$ satisfying (b1)--(b5). By repeating this $3k-1$ times, we get $B_{1},\dots,B_{3k-1}$ and $b_1,\dots, b_{3k-1}$. We let $B:=\bigcup_{i=1}^{3k-1} B_i.$ 
Note that $T[B_i]$ is a transitive tournament for each $i\in [3k-1]$. For each $i$, we let $B'_i$ be the set of the last $\max(\lceil |B_i|/5-13\rceil , 0)$ vertices, and let $B''_i$ be the set of the first $\min(\lceil 5 \log(k)+30 \rceil , |B_i|)$ vertices in the transitive ordering of $T[B_i]$, respectively. Note that $B_i'$ and $B_i''$ are not necessarily disjoint. 

We define 
\begin{align*}
A_{\rm sink}:=\{a_1,\dots, a_{3k-1}\}, B_{\rm source}:=\{b_1,\dots, b_{3k-1}\}, B':= \bigcup_{i=1}^{3k-1} B'_i, \text{ and }B'':= \bigcup_{i=1}^{3k-1} B''_i.
\end{align*}
From this construction, we get numbers $d^+_{1},\dots, d^+_{3k-1},d^-_1,\dots, d^-_{3k-1}$ satisfying
\begin{align}\label{d size}
\delta^+ \geq d^+_1 \geq d^+_2 \geq \dots \geq d^+_{3k-1} \enspace \text{ and } \enspace \delta^-\geq d^-_1 \geq d^-_2 \geq \dots \geq d^-_{3k-1},
\end{align}
and sets $A_1,\dots, A_{3k-1}$, $B_1,\dots, B_{3k-1}$, $B'_1,\dots, B'_{3k-1}$, $B''_1,\dots, B''_{3k-1}$ and vertices $a_1,\dots$, $ a_{3k-1}$, $b_{1},\dots, b_{3k-1}$ satisfying the following (A1)--(A3) and (B1)--(B6) for all $i\in [3k-1]$.
\begin{enumerate}
\item[(A1)] $\frac{1}{2}\log(d^+_{i}+1)+1 \leq |A_i|\leq \frac{5}{2}\log(d^+_{i}+1)+2,$
\item[(A2)] $T[A_i]$ is a transitive tournament with source $x_i$ and sink $a_i$,
\item[(A3)] $A_i$ in-dominates $V\setminus(A\cup B)$,

\item[(B1)] $\frac{1}{2}\log(d^-_{i}+1)+1 \leq |B_i|\leq \frac{5}{2}\log(d^-_{i}+1)+2,$
\item[(B2)] $T[B_i]$ is a transitive tournament with sink $y_i$ and source $b_i$,
\item[(B3)] $B_i$ out-dominates $V\setminus(A\cup B)$,
\item[(B4)] $|B'_i|\geq |B_i|/5-13$ and for $v \in B'_i$ we have $$| N^+_T(v)\setminus (A\cup \bigcup_{j=1}^{i}B_j) |\geq 8 (d^{-}_{i})^{1/7} - 1 ,\enspace | N^-_T(v)\setminus (A\cup \bigcup_{j=1}^{i}B_j) |\geq 8 (d^{-}_{i})^{1/7} - 1.$$

\item[(B5)] $|B''_i| < 5\log(k)+31$ and for $v\in B_i\setminus B''_i$ we have $$| N^+_T(v)\setminus (A\cup \bigcup_{j=1}^{i}B_j) | \geq 1000k^2 ,\enspace | N^-_T(v)\setminus (A\cup \bigcup_{j=1}^{i}B_j) | \geq 1000k^2.$$

\item[(B6)] For any vertex $v \in B_i\setminus B'_i$, we have $B'_i\subseteq N^+_{T}(v)$.
\end{enumerate}\vspace{0.3cm}

 By Lemma~\ref{degree}, each of $T[A_{\rm sink}]$ and $T[B_{\rm source}]$ contains $k$ vertices of in-degree at least $k$ and $k$ vertices of out-degree at least $k$. Let $a_{i_1},\dots, a_{i_k} \in A_{\rm sink}$ be $k$ distinct vertices having in-degree at least $k$ in $T[A_{\rm sink}]$ and let $b_{j_1},\dots, b_{j_k} \in B_{\rm source}$ be distinct $k$ vertices having out-degree at least $k$ in $T[B_{\rm source}]$.
By (A1), (B1) and the fact that $\delta^- \leq n-1$, we have $|A\cup B| \leq (6k-2)(\frac{5}{2}\log(n)+2) < n-k$ since $n\geq 100 k\log(k+1)$ and $k \geq 2$. Thus we have 
\begin{align}\label{V-A-B size}
|V\setminus(A\cup B)| \geq k.
\end{align}

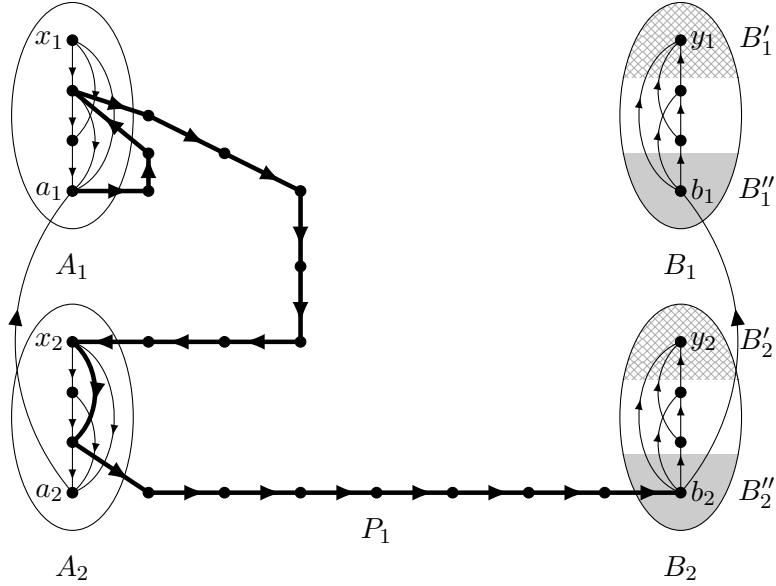
\begin{figure}
\centering
\begin{tikzpicture}[scale=1]

\begin{scope}
\clip
  (6,2) rectangle (10,0.5);
  \draw[preaction={fill=none}, pattern=crosshatch, pattern color=black!30](8,0) ellipse [x radius=0.8,y radius=1.5];
\end{scope}

\begin{scope}
\clip
  (6,0.5) rectangle (10,-0.5);
  \draw[fill=none] (8,0) ellipse [x radius=0.8,y radius=1.5];
\end{scope}

\begin{scope}
\clip
  (6,-2) rectangle (10,-0.5);
  \draw[fill=black!20] (8,0) ellipse [x radius=0.8,y radius=1.5];
\end{scope}

\node at (8,-2) {$B_1$};
\node at (9,1) {$B'_1$};
\node at (9,-1) {$B''_1$};

\begin{scope}
\clip
  (6,-2) rectangle (10,-3.5);
  \draw[preaction={fill=none}, pattern=crosshatch, pattern color=black!30](8,-4) ellipse [x radius=0.8,y radius=1.5];
\end{scope}

\begin{scope}
\clip
  (6,-3.5) rectangle (10,-4.5);
  \draw[fill=none] (8,-4) ellipse [x radius=0.8,y radius=1.5];
\end{scope}

\begin{scope}
\clip
  (6,-6) rectangle (10,-4.5);
  \draw[fill=black!20] (8,-4) ellipse [x radius=0.8,y radius=1.5];
\end{scope}

\node at (8,-6) {$B_2$};
\node at (9,-3) {$B'_2$};
\node at (9,-5) {$B''_2$};

 \draw[fill=none] (0,0) ellipse [x radius=0.8,y radius=1.5];
\node at (0,-2) {$A_1$};
 \draw[fill=none] (0,-4) ellipse [x radius=0.8,y radius=1.5];
\node at (0,-6) {$A_2$};

\filldraw[fill=black] (0,1) circle (2pt);
\filldraw[fill=black] (0,0.33) circle (2pt);
\filldraw[fill=black] (0,-0.33) circle (2pt);
\filldraw[fill=black] (0,-1) circle (2pt);

\draw[middlearrow={latex}] (0,1) to[bend left=70] (0,-1) ;
\draw[middleuparrow={latex}] (0,1) -- (0,0.33) ;
\draw[middleuparrow={latex}] (0,0.33) -- (0,-0.33) ;
\draw[middleuparrow={latex}] (0,-0.33) -- (0,-1) ;
\draw[middlearrow={latex}] (0,1) to[bend left=50] (0,-0.33) ;
\draw[middlearrow={latex}] (0,0.33) to[bend left=50](0,-1) ;

\node at (-0.3,-1) {$a_1$};
\node at (-0.3,1) {$x_1$};

\filldraw[fill=black] (0,-3) circle (2pt);
\filldraw[fill=black] (0,-3.67) circle (2pt);
\filldraw[fill=black] (0,-4.33) circle (2pt);
\filldraw[fill=black] (0,-5) circle (2pt);

\draw[middlearrow={latex}] (0,-3) to[bend left=70] (0,-5) ;
\draw[middleuparrow={latex}] (0,-3) -- (0,-3.67) ;
\draw[middleuparrow={latex}] (0,-3.67) -- (0,-4.33) ;
\draw[middleuparrow={latex}] (0,-4.33) -- (0,-5) ;
%\draw[middlearrow={latex}] (0,-3) to[bend left=50] (0,-4.33) ;
\draw[middlearrow={latex}] (0,-3.67) to[bend left=50](0,-5) ;

\node at (-0.3,-5) {$a_2$};
\node at (-0.3,-3) {$x_2$};

\filldraw[fill=black] (8,1) circle (2pt);
\filldraw[fill=black] (8,0.33) circle (2pt);
\filldraw[fill=black] (8,-0.33) circle (2pt);
\filldraw[fill=black] (8,-1) circle (2pt);

\draw[middlearrow={latex}] (8,-1) to[bend left=70] (8,1) ;
\draw[middleuparrow={latex}] (8,0.33) -- (8,1);
\draw[middleuparrow={latex}] (8,-0.33) -- (8,0.33) ;
\draw[middleuparrow={latex}] (8,-1) -- (8,-0.33) ;
\draw[middlearrow={latex}] (8,-0.33) to[bend left=50] (8,1) ;
\draw[middlearrow={latex}] (8,-1) to[bend left=50] (8,0.33) ;
\node at (8.3,-1) {$b_1$};
\node at (8.3,1) {$y_1$};

\filldraw[fill=black] (8,-3) circle (2pt);
\filldraw[fill=black] (8,-3.67) circle (2pt);
\filldraw[fill=black] (8,-4.33) circle (2pt);
\filldraw[fill=black] (8,-5) circle (2pt);

\draw[middlearrow={latex}] (8,-5) to[bend left=70] (8,-3) ;
\draw[middleuparrow={latex}] (8,-3.67) -- (8,-3);
\draw[middleuparrow={latex}] (8,-4.33) -- (8,-3.67) ;
\draw[middleuparrow={latex}] (8,-5) -- (8,-4.33) ;
\draw[middlearrow={latex}] (8,-4.33) to[bend left=50] (8,-3) ;
\draw[middlearrow={latex}] (8,-5) to[bend left=50] (8,-3.67) ;
\node at (8.3,-5) {$b_2$};
\node at (8.3,-3) {$y_2$};

\draw[middlearrow={triangle 45}] (0,-5) to[bend left=40] (0,-1) ;

\draw[middlearrow={triangle 45}] (8,-5) to[bend right=40] (8,-1) ;

\filldraw[fill=black] (1,-1) circle (2pt);
\filldraw[fill=black] (1,-0.5) circle (2pt);
\filldraw[fill=black] (1,0) circle (2pt);
\filldraw[fill=black] (2,-0.5) circle (2pt);

\filldraw[fill=black] (3,-1) circle (2pt);
\filldraw[fill=black] (3,-2) circle (2pt);
\filldraw[fill=black] (3,-3) circle (2pt);
\filldraw[fill=black] (2,-3) circle (2pt);
\filldraw[fill=black] (1,-3) circle (2pt);

\filldraw[fill=black] (1,-5) circle (2pt);
\filldraw[fill=black] (2,-5) circle (2pt);
\filldraw[fill=black] (3,-5) circle (2pt);
\filldraw[fill=black] (4,-5) circle (2pt);
\filldraw[fill=black] (5,-5) circle (2pt);
\filldraw[fill=black] (6,-5) circle (2pt);
\filldraw[fill=black] (7,-5) circle (2pt);

\draw[middleuparrow={latex}, line width=0.6mm] (0,-1) -- (1,-1);
\draw[middleupuparrow={latex}, line width=0.6mm] (1,-1) -- (1,-0.5);
\draw[middlearrow={latex}, line width=0.6mm] (1,-0.5) -- (0,0.33);
\draw[middleuparrow={latex}, line width=0.6mm] (0,0.33) -- (1,0);
\draw[middleuparrow={latex}, line width=0.6mm] (1,0) -- (2,-0.5);
\draw[middleuparrow={latex}, line width=0.6mm] (2,-0.5) -- (3,-1);
\draw[middleuparrow={latex}, line width=0.6mm] (3,-1) -- (3,-2);
\draw[middleuparrow={latex}, line width=0.6mm] (3,-2) -- (3,-3);
\draw[middleuparrow={latex}, line width=0.6mm] (3,-3) -- (2,-3);
\draw[middleuparrow={latex}, line width=0.6mm] (2,-3) -- (1,-3);
\draw[middleuparrow={latex}, line width=0.6mm] (1,-3) -- (0,-3);
\draw[middlearrow={latex}, line width=0.6mm] (0,-3) to[bend left=50] (0,-4.33);
\draw[middleuparrow={latex}, line width=0.6mm] (0,-4.33) -- (1,-5);
\draw[middleuparrow={latex}, line width=0.6mm] (1,-5) -- (2,-5);
\draw[middleuparrow={latex}, line width=0.6mm] (2,-5) -- (3,-5);
\draw[middleuparrow={latex}, line width=0.6mm] (3,-5) -- (4,-5);
\draw[middleuparrow={latex}, line width=0.6mm] (4,-5) -- (5,-5);
\draw[middleuparrow={latex}, line width=0.6mm] (5,-5) -- (6,-5);
\draw[middleuparrow={latex}, line width=0.6mm] (6,-5) -- (7,-5);
\draw[middleuparrow={latex}, line width=0.6mm] (7,-5) -- (8,-5);

\node at (4,-5.5) {{\bf $P_1$}};

\end{tikzpicture}
\caption{ A picture when $k=1, i_1 = 1$ and $j_1 =2$. }
\end{figure}

Our aim is to find collections of arcs $E_0, E_1, E_2, E_3, E_4$ and $E_5$ which together form a desired digraph $D$. Since the tournament $T$ is strongly $k$-connected, by Menger's theorem,
let $P_1 , \dots , P_k$ be $k$ vertex-disjoint paths from $\left \{a_{i_1} , \dots , a_{i_k} \right \}$ to $\{b_{j_1},\dots, b_{j_k}\}$. We choose those $k$ vertex-disjoint paths with the minimum length $\sum_{i=1}^{k}{|E(P_i)|}$, and thus each path $P_i$ is backwards-transitive for $1 \leq i \leq k$. 
Note that $V(P_i)$ is not necessarily disjoint from $A\cup B \setminus \{a_{i_1},\dots, a_{i_k}, b_1,\dots, b_{j_k}\}$.
By permuting indices, we may assume that $P_s$ is a backwards-transitive path from $a_{i_s}$ to $b_{j_s}$. See Figure~2 for the picture which we currently have.
Let $V^{\rm int}(P_s)$ be the set of internal vertices of $P_s$. We define 
\begin{align}\label{V1 def}
V_1:= (A\cup B)  \setminus (\bigcup_{i=1}^{k} V^{\rm int}(P_i)), \enspace  V'_1 := (A\cup B)\cap (\bigcup_{i=1}^{k} V^{\rm int}(P_i))~ \text{ and } ~E_0 := \bigcup_{s=1}^{k} E(P_s).
\end{align}

Before starting the construction of $E_1,E_2, E_3, E_4$ and $E_5$, we prove Claim~\ref{in-AB} and Claim~\ref{out-AB} showing that for any $v\in A\cup B$ there exists a $k$-fan from $v$ to $V\setminus (A\cup B)$ and a $k$-fan from $V\setminus (A\cup B)$ to $v$ consisting of short paths.
\begin{claim} \label{in-AB}
For any vertex $v\in A\cup B$, we can find a $k$-fan $\{P^-(v,1),\dots, P^-(v,k)\}$ from $V\setminus(A\cup B)$ to $v$ such that $\sum_{i=1}^{k} |E(P^-(v,i))| \leq 70k\log(k+1)$.
\end{claim}

\begin{proof}[Proof of Claim~\ref{in-AB}]
Note that \eqref{delta size}, \eqref{d size}, (A1) and (B1) together imply that
\begin{align}\label{AB size}
|A\cup B| \leq (6k-2)(\frac{5}{2}\log(\delta^-+1)+2).
\end{align}
 We consider the following two cases. \newline

\noindent {\bf Case 1}. $\delta^- \leq 60k^2$.

In this case, consider $\{P^-(v,1),\dots, P^-(v,k)\}$, a $k$-fan from $V\setminus(A\cup B)$ to $v$. Such a $k$-fan exists because of Fact~\ref{Menger} and \eqref{V-A-B size}. By \eqref{AB size}, we have $|A\cup B| \leq (6k-2)(\frac{5}{2}\log(60k^2+1)+2) \leq 69 k\log(k+1).$ Since every vertex in each $P^-(v,i)$ is in $A \cup B$ except for one vertex, we have $\sum_{i=1}^{k}{|E(P^-(v,i))|} \leq |A \cup B| + k \leq 70k \log (k+1)$.\\

\noindent {\bf Case 2}. $\delta^- > 60k^2$. 

Since $k\geq 2$, we have
$$\delta^- \geq (6k-2)(\frac{5}{2}\log(\delta^-+1)+2) +2k \stackrel{\eqref{AB size}}{\geq}  |A\cup B|+2k.$$
Thus for any vertex $u \notin Y$, we have $d^-(u) \geq \delta^- \geq |A\cup B|+2k$. 

If $v \notin Y$, take $k$ distinct paths of length $1$ from $V\setminus(A\cup B)$ to $v$, and let $P^-(v,1),\dots,P^-(v,k)$ be those paths of length $1$. Then we have  $\sum_{i=1}^{k}{|E(P^-(v,i))|} \leq k \leq 70k\log(k+1)$.
If $v\in Y$, then take $\{Q_1,\dots, Q_k\}$, a $k$-fan from $V\setminus Y$ to $v$ given by Fact~\ref{Menger} and \eqref{V-A-B size}. Let $v_i$ be the starting vertex of $Q_i$ for $1 \leq i \leq k$. Then we have
$$\sum_{i=1}^{k} |E(Q_i)| \leq |Y|+k \leq 4k-1.$$

Consider $i\in[k]$ with $v_i \in A\cup B$. Since each $v_i$ is not in $Y$, $d_T^{-}(v_i) \geq \delta^- \geq |A \cup B| + 2k$ and $v_i$ has at least $2k$ in-neighbors outside $A \cup B$. For each $i\in [k]$ with $v_i \in A\cup B$, we choose $v'_i$ in $N^-_{T}(v_i) \setminus (A\cup B\cup \{v_1,\dots, v_k\})$ in the way that $v'_i$s are all distinct. 
Let 
$$P^-(v,i):= \left\{ \begin{array}{ll} Q_i\cup \{ \overrightarrow{v'_i v_i}\} & \text{if }v_i\in A\cup B, \\
Q_i &\text{if }v_i\notin A\cup B. \end{array}\right.$$
Then the paths $P^-(v,1),\dots, P^-(v,k)$ form a $k$-fan from $V\setminus(A\cup B)$ to $v$ such that
$$\sum_{i=1}^{k} |E(P^-(v,i))|\leq k + \sum_{i=1}^{k} |E(Q_i)| \leq |Y| + 2k = 5k-1\leq 70 k\log(k+1).$$
This proves Claim~\ref{in-AB}.
\end{proof}

\begin{claim}\label{subclaim}
For each $v\in A\cup B''$, there exists a $k$-fan $\{P^+_*(v,1),\dots, P^+_*(v,k)\}$ from $v$ to $V\setminus(A\cup B'')$ such that
$\sum_{i=1}^{k} |E(P^+_*(v,i))| \leq 98 k\log(k+1).$
\end{claim}
\begin{proof}[Proof of Claim~\ref{subclaim}]
Note that we have
 \begin{eqnarray}\label{AB'' size}
|A\cup B''|   &\stackrel{(\text{A}1)}{\leq}& \sum_{i=1}^{3k-1} (\frac{5}{2}\log(d^+_i+1) +2) +|B''| \nonumber \\
   &\stackrel{\eqref{d size},(\text{B}5)}{<} & (3k-1)( \frac{5}{2} \log(\delta^+ +1) +2) + (3k-1)(5 \log(k)+31)
 \end{eqnarray}
To prove Claim~\ref{subclaim}, we consider the following two cases. \\

\noindent {\bf Case 1}. $\delta^+ \leq 100 k^2$. 

 Since $T$ is strongly $k$-connected, there exists $\{P^+_*(v,1),\dots, P^+_*(v,k)\}$, a $k$-fan from $v$ to $V\setminus(A\cup B'')$ by Fact~\ref{Menger} and \eqref{V-A-B size}. Since $P^+_*(v,1),\dots, P^+_*(v,k)$ contains at most $k$ vertices outside $A\cup B''$ and $\delta^+ \leq 100 k^2$, we have
 \begin{eqnarray*}
 \sum_{i=1}^{3k-1}|E(P^+_* (v,i))| \leq  |A\cup B''|+ k 
   \stackrel{\eqref{AB'' size}}{\leq}  98 k \log(k+1).
 \end{eqnarray*}

\noindent {\bf Case 2}. $\delta^+ \geq 100 k^{2}$.

In this case, we have
\begin{eqnarray*}
|A\cup B''| + 2k &\stackrel{\eqref{AB'' size}}{<}& (3k-1)( \frac{5}{2} \log(\delta^+ +1) +2) + (3k-1)(5\log(k)+ 31) + 2k \leq  \delta^+
\end{eqnarray*}

If $v\notin X$, then $d^{+}_T(v) \geq \delta^+ \geq |A\cup  B''|+ 2k$. So we can find $k$ paths $Q'_1,\dots, Q'_k$ of length $1$ from $v$ to $V\setminus (A\cup B'')$. Let $P^+_*(v,1),\dots, P^+_*(v,k)$ be those paths of length $1$. Then $\sum_{i=1}^{k}{|E(P^+_* (v,i))|} \leq k \leq 98 k \log (k+1)$.

If $v\in X$, then we find a $k$-fan $\{Q'_1,\dots, Q'_k\}$ from $v$ to $V\setminus X$ by Fact~\ref{Menger} and \eqref{V-A-B size}.
Then because all vertices of $Q'_i$ except the last vertex belong to $X$, we have $\sum_{i=1}^{k}|E(Q'_i)| \leq |X|+k$. Let $u'_i$ be the end vertex of $Q'_i$, for $1 \leq i \leq k$. Consider $i\in [k]$ with $u'_i \in A\cup B''$. Since $u'_i \notin X$ and $d_T^{+}(u'_i) \geq \delta^+ \geq |A \cup B''| + 2k$, $u'_i$ has at least $2k$ out-neighbors in $V\setminus (A\cup B'')$, we can choose $u''_i \in N_{T}^+(u'_i)\setminus (A\cup B''\cup\{u'_1,\dots, u'_k\})$ such that $u''_i$s are distinct. We let
$$P^+_*(v,i):= \left\{ \begin{array}{ll} Q'_i\cup \{ \overrightarrow{u'_i u''_i}\} & \text{if }u'_i\in A\cup B'', \\
Q'_i &\text{if }u'_i\notin A\cup B''. \end{array}\right.$$
 Then we have a $k$-fan $\{P^+_*(v,1),\dots, P^+_*(v,k)\}$ from $v$ to $V\setminus(A\cup B'')$ such that  $$ \sum_{i=1}^{k} |E(P^+_*(v,i))| \leq \sum_{i=1}^{k}|E(Q'_i)| + k \leq |X|+2k = 5k-1 \leq 98 k\log(k+1).$$ 
 This proves Claim~\ref{subclaim}.
\end{proof}

Now we prove Claim~\ref{out-AB} by using Claim~\ref{subclaim}.

\begin{claim} \label{out-AB}
For any vertex $v\in A\cup B$, there exists a $k$-fan $\{P^+(v,1),\dots, P^+(v,k) \}$ from $v$ to $V\setminus(A\cup B)$ with 
$ \sum_{i=1}^{k} |E(P^+(v,i))| \leq 100k\log(k+1).$
\end{claim}
\begin{proof}[Proof of Claim~\ref{out-AB}]
We first use Claim~\ref{subclaim} to find a $k$-fan from $v$ to $V\setminus(A\cup B'')$ such that $\sum_{i=1}^{k} |E(P^+_*(v,i))|\leq 98k\log(k+1).$ Let $u_i$ be the last vertex in $P^+_*(v,i)$ and let $U:=\{u_1,\dots,u_k\}.$ Then for each $i\in [k]$ all vertices in $P^+_*(v,i)$ except $u_i$ belong to $A\cup B''$, and $u_i$ is either in $V\setminus (A\cup B)$ or in $B\setminus B''$. For each $i$ with $u_i \in B\setminus B''$, let $\ell_i$ be the index such that $u_i \in B_{\ell_i}$. Then we can partition $[k]$ into four sets $I_1,I_2,I_3$ and $I_4$ as follows.
\begin{align*}
&\text{For } i\in I_1, \text{ we have } |B_{\ell_i}|\geq 18k+80, u_i \in B \setminus B'' \text{ and } u_i\notin B'_{\ell_i}, \nonumber\\
&\text{for } i\in I_2, \text{ we have } |B_{\ell_i}| \geq 18k+80, u_i \in B \setminus B'' \text{ and } u_i \in B'_{\ell_i}, \nonumber \\
&\text{for } i\in I_3, \text{ we have } |B_{\ell_i}| < 18k+80 \text{ and } u_i \in B \setminus B'' , \nonumber\\
&\text{for } i\in I_4, \text{ we have } u_i\notin A\cup B.
\end{align*}
First, consider $i\in I_1 \cup I_2$. Since $|B_{\ell_i}|\geq 18k+80$, (B1) implies that 
\begin{align}\label{d- min}
d^-_{\ell_i} \geq 2^{\frac{2}{5}(|B_{\ell_i}|-2)}-1 \geq 2^{7k+30}.
\end{align}

For any $u\in B'_{\ell_i}$ we have
\begin{eqnarray}\label{3k out neighbor}
| N^+_T(u)\setminus (A\cup B) |
 &\geq& \left | N^+_T(u)\setminus (A\cup \bigcup_{p=1}^{{\ell_i}}B_{p}) \right | -  \left |\bigcup_{p=\ell_i+1}^{3k-1} B_p  \right | \nonumber\\
  &\stackrel{(\text{B}4)}{\geq}& 8 (d^{-}_{\ell_i})^{1/7} - 1 - \left |\bigcup_{p=\ell_i+1}^{3k-1} B_p  \right | \nonumber \\
&\stackrel{\eqref{d- min}}{\geq} & (3k-1) (\frac{5}{2}\log(d^{-}_{\ell_i}+1) +2) + 3k - \left |\bigcup_{p=\ell_i+1}^{3k-1} B_p  \right |\nonumber \\
&\stackrel{(\text{B}1),\eqref{d size}}{\geq}& 3k.
\end{eqnarray}

Here, we get the third inequality since $8x^{1/7}-1 \geq (3k-1)(\frac{5}{2} \log(x+1)+2)+3k$ holds for $x\geq 2^{7k+30}$ and $k \geq 2$. Thus any vertex $u \in B'_{\ell_i}$ has at least $3k$ out-neighbors in $V\setminus (A\cup B)$.

For $i\in I_1$, (B4) implies that $|B'_{\ell_i}| \geq |B_{\ell_i}|/5-13 \geq 3k$ and (B6) implies that $B'_{\ell_i} \subseteq N^+_{T}(u_i)$. From this we obtain $|(N^+_{T}(u_i) \cap B'_{\ell_i}) \setminus U| = |B'_{\ell_i} \setminus U| \geq 3k -k \geq 2k.$ Thus we can choose a set $W=\{ w_i : i\in I_1\}$ of $|I_1|$ distinct vertices such that $w_i\in N_T^+ (u_i) \cap (B'_{\ell_i} \setminus U)$. 
Again, \eqref{3k out neighbor} implies that 
$$|N^+_{T}(w_i)\setminus (A\cup B \cup U \cup W)|\geq k,$$ so we can further choose a set $W'=\{ w'_i : i\in I_1\}$ of $|I_1|$ distinct vertices such that $w'_i \in N^+_{T}(w_i)\setminus (A\cup B\cup U \cup W)$. 

Now we consider $i\in I_2$. In this case $u_i \in B'_{\ell_i}$ and \eqref{3k out neighbor} imply that 
$$|N^+_{T}(u_i)\setminus (A\cup B \cup U \cup W\cup W')|\geq 2k - 2|I_1| \geq |I_2|,$$ so we can further choose a set $W^*= \{ w^*_i : i\in I_2\}$ of $|I_2|$ distinct vertices such that $w^*_i \in N^+_{T}(u_i)\setminus (A\cup B\cup U \cup W\cup W')$. 

Now we consider $i\in I_3$. In this case, $u_i$ belongs to $B_{\ell_i}\setminus B''_{\ell_i}$. Thus
\begin{eqnarray*}
\left |N^+_{T}(u'_i) \setminus (A\cup B ) \right | & \geq &
\left |N^+_{T}(u'_i) \setminus (A\cup \bigcup_{p=1}^{\ell_i} B_p) \right | - \left | \bigcup_{p= \ell_i+1}^{3k-1} B_p \right | \\
& \stackrel{(\text{B}1),(\text{B}5)}{\geq} & 1000k^2 - \sum_{p=\ell_i+1}^{3k-1} (\frac{5}{2} \log(d^{-}_p+1)+2) \\
&\stackrel{\eqref{d size}}{\geq} &  1000k^2 - (3k-1)(\frac{5}{2} \log(d^{-}_{\ell_i}+1)+2) \\
&\stackrel{(\text{B}1)}{\geq}& 1000k^2 - 5(3k-1)|B_{\ell_i}|  \\
&\geq&  1000 k^2 - 5(3k-1)(18k+80)  \geq 5k \geq |I_3|+4k.
\end{eqnarray*}

Thus we can choose a set $W^{**}:=\{w^{**}_i : i\in I_3\}$ of $|I_3|$ distinct vertices such that $w^{**}_i \in N^+_{T}(u_i)\setminus (A\cup B \cup U\cup W \cup W' \cup W^*)$. Note that $U,W,W',W^*,W^{**}$ are pairwise disjoint sets by construction. For $i \in [k]$, let $P^+ (v,i)$ be a path from $v$ to $V \setminus (A \cup B)$ as follows.
$$
E(P^+(v,i)) := \left\{ \begin{array}{ll} 
E(P^+_*(v,i))\cup \{ \overrightarrow{u_i w_i}, \overrightarrow{w_i w'_i}\} &\text{if }i\in I_1,\\
E(P^+_*(v,i))\cup \{ \overrightarrow{u_i w^*_i}\} &\text{if }i\in I_2, \\
E(P^+_*(v,i))\cup \{ \overrightarrow{u_i w^{**}_i}\} &\text{if }i\in I_3, \\
E(P^+_*(v,i)) &\text{if }i\in I_4. \\
\end{array}\right.
$$

We claim that $\left \{ P^+ (v,i) \right \}_{i=1}^{k}$ is a $k$-fan from $v$ to $V \setminus (A \cup B)$, and the sum of lengths is small. Indeed, for any $i\in [k]$, $P^+(v,i)$ is a path from $v$ to $V\setminus (A\cup B)$. Note that paths $\left \{V(P^+(v,i)) \right \}_{i=1}^{k}$ form a $k$-fan since the paths $\left \{V(P^+_*(v,i))\setminus \{v\} \right \}_{i=1}^{k}$ are pairwise-disjoint, and $U,W,W',W^*,W^{**}$ are pairwise disjoint. Moreover,
$$\sum_{i=1}^{k} |E(P^+(v,i))| = \sum_{i=1}^{k}|E(P^+_*(v,i))| + 2|I_1|+|I_2|+|I_3| \leq 98 k\log(k+1) + 2k \leq 100 k\log(k+1).$$
This proves Claim~\ref{out-AB}.
\end{proof}

Recall that $V_1, V'_1$ and $E_0$ are defined in \eqref{V1 def} and note that we have $\{a_{i_1},\dots, a_{i_k}, b_{j_1},\dots, b_{j_k}\}\subseteq V_1$.
Now we will find a set of arcs $E_1$ as in the following claim.
\begin{claim}\label{E1}
There exist a set of arcs $E_1\subseteq E(T)$ and a set of vertices $V_2\subseteq V\setminus  (A\cup B)$ satisfying the following.
\begin{enumerate} 
\item[$({\rm E}1)_1$] $|E_1| \leq k|V_1|+(k-1)|V'_1| + 680 k^2\log(k+1)$ and $ |V_2| \leq 8k^2$.
\item[$({\rm E}1)_2$] For any set $S\subseteq V(T)$ of size $k-1$ and a vertex $v\in (V_1 \cup V_1 ')\setminus S$, we can find a path $P$ in $T-S$ from $v$ to $V_2$ such that $E(P)\subseteq E_0\cup E_1$.
\item[$({\rm E}1)_3$] For any set $S\subseteq V(T)$ of size $k-1$ and a vertex $v\in (V_1 \cup V_1 ')\setminus S$, we can find a path $P$ in $T-S$ from $V_2$ to $v$ such that $E(P)\subseteq E_0\cup E_1$.
\end{enumerate}
\end{claim}
\begin{proof}[Proof of Claim~\ref{E1}]
We apply Lemma~\ref{main lemma} to $T[V_1]$ with parameters $0,k$ corresponding to $s,k$, respectively. Then we obtain an ordering $\sigma_1$ of $V_1$ with a $(\sigma_1,k,2k-1)$-good digraph $D_1\subseteq T[V_1]$ such that $|E(D_1)|\leq k|V_1| - k$. 
We also consider a digraph $T[V'_1]- E_0$. Since $\delta(T[V'_1]- E_0) \geq |V'_1|-3$, we can apply Lemma \ref{main lemma} to $T[V'_1]-E_0$ with parameters $2,(k-1)$ corresponding to $s,k$, respectively. Then we obtain an ordering $\sigma'_1$ of $V'_1$ and a $(\sigma'_1,k-1,2k-1)$-good digraph $D'_1 \subseteq T[V'_1]-E_0$ with $|E(D'_1)| \leq (k-1)|V'_1|+(k-1).$  
Here, it is important to take $(\sigma_1', k-1, 2k-1)$-good subgraph of $T[V_1'] - E_0$ instead of $(\sigma_1', k, 2k-1)$-good subgraph of $T[V_1']$, otherwise we would get $|E(D'_1)|\leq k|V'_1|+k$ which is too much for our purpose.

Now we define $W_1^-$ and $W_1^+$ as follows.
$$W_1^-:= \sigma_1(1,2k-1) \cup \sigma'_1(1,2k -1 )\text{ and } W_1^+:= \sigma_1(|V_1|-2k+1,|V_1|)\cup \sigma'_1(|V'_1| - 2k+1,|V'_1|)$$
This gives
\begin{align}\label{W1 size}
|W_1^-| , |W_1^+| \leq 4k-2.
\end{align}
For each vertex $u \in W_1^-$ we use Claim~\ref{in-AB} to obtain a $k$-fan $\left \{P^-(u,1),\dots,P^-(u,k) \right \}$ in $T$ from $V\setminus(A\cup B)$ to $u$ with 
\begin{align}\label{1 in path length}
\sum_{i=1}^{k} |E(P^-(u,i))| \leq 70k\log(k+1).
\end{align}
For each vertex $u\in W_1^+$, we use Claim~\ref{out-AB} to obtain a $k$-fan $\left \{P^+(u,1),\dots,P^+(u,k) \right \}$ in $T$ from $u$ to $V\setminus (A\cup B)$ with 
\begin{align}\label{1 out path length}
\sum_{i=1}^{k} |E(P^+(u,i))| \leq 100k\log(k+1).
\end{align}
Let 
\begin{align}
E_1 := E(D_1)\cup E(D'_1) \cup \bigcup_{u\in W_1^-, i\in [k]} E(P^-(u,i)) \cup \bigcup_{u\in W_1^+, i\in [k]} E(P^+(u,i)),
\end{align}
\begin{align*}
V_2 := V(E_1)\setminus (V_1\cup V_1').
\end{align*}
Since $V_1 \cup V_1' = A \cup B$, every vertex in $V_2$ is either one of the last vertices of $P^+(u,i)$ for some $i\in [k]$ and $u \in W_1^+$ or one of the first vertex of $P^-(u,i)$ for some $i\in [k]$ and $u\in W_1^-$. Thus we have
$|V_2| \leq k(|W_1^+|+|W_1^-|) \stackrel{\eqref{W1 size}}{\leq} 8k^2 $. Moreover,
\begin{eqnarray*}
|E_1| & \stackrel{\eqref{1 in path length},\eqref{1 out path length}}{\leq}& |E(D_1)|+|E(D_2)| + 70k\log(k+1)|W_1^-| + 100k\log(k+1)|W_1^+| \\
& \stackrel{\eqref{W1 size}}{\leq}& k|V_1|+ (k-1)|V'_1| + 680 k^2\log(k+1).
\end{eqnarray*}

This proves $({\rm E}1)_1$. To prove $({\rm E}1)_2$, let $S$ be a set of $k-1$ vertices in $V$ and let $v$ be a vertex with $v \in (V_1 \cup V_1 ') \setminus S$. We consider the following two cases. \newline

\noindent {\bf Case 1}. $v\in V_1$. 

 By Claim \ref{connectivity claim} and the fact that $D_1$ is $(\sigma_1,k,2k-1)$-good, we can find a path $P'$ from $v$ to a vertex $u\in W_1^+$ in $T-S$ such that $E(P') \subseteq E_1$. Also $P^+(u,1),\dots, P^+(u,k)$ are disjoint paths except the common starting vertex $u \notin S$, thus there exists $j\in [k]$ such that $P^+(u,j)$ does not intersect with $S$. 
Then $E(P')\cup E(P^+(u,j))$ contains a path $P$ in $T-S$ from $v$ to $V_2$ with $E(P)\subseteq E_1$. \newline

\noindent {\bf Case 2}. $v\in V'_1$.

 Assume $\sigma'_1=(v'_1,\dots, v'_{|V'_1|}).$ We consider the maximum index $i$ such that there is a path $P'$ from $v$ to $v_i'$ in $D'_1 - S$. If $i\geq |V'_1|-2k+2$, then we have $v'_i \in W^+_1$ and we can choose $j\in [k]$ such that $P^+ (v'_i , j)$ does not intersect with $S$. Then $E(P')\cup E(P^+(v'_i,j))$ contains a path $P$ in $T-S$ from $v$ to $V_2$ with $E(P)\subseteq E_1$.
If $i<|V'_1|-2k+2$, then the maximality of $i$ implies $N^+_{D'_1}(v'_i) \subseteq S$ by (D1) and the fact that $D'_1$ is $(\sigma'_1,k-1,2k-1)$-good. Since 
$$k-1 \stackrel{({\rm D}2)}{\leq} |N^+_{D'_1}(v'_i)|\leq |S|= k-1,$$ we have
\begin{align}\label{S N+}
S=N^+_{D'_1}(v'_i).
\end{align} 

 By \eqref{V1 def} and the fact that $v'_i \in V'_1$, there exists $s\in [k]$ such that $v'_i \in V^{\rm int}(P_s)$. We let $P''$ be the sub-path of $P_s$ from $v'_i$ to $b_{j_s}$. Since $P_s$ is backwards-transitive, every vertex in $V(P'')$ belongs to $N^-_T(v'_i)$ except the first vertex $v'_i$ and the second vertex, say $u'$, of $P''$. Since $\overrightarrow{v'_iu'}\in E(P_s) \subseteq E_0$ and $D_1'\subseteq T[V'_1]-E_0$, we obtain $\overrightarrow{v'_iu'}\notin E(D'_1)$. Thus $u'\notin N^+_{D'_1}(v'_i)$. 
This with the fact that $V(P'')\subseteq N^-_T(v'_i)\cup \{v'_i , u'\}$ implies that
 $$V(P'') \cap S \subseteq (N^-_T(v'_i)\cup \{v'_i , u'\})\cap S \stackrel{\eqref{S N+}}{=} (N^-_T(v'_i) \cup \{v'_i , u'\})\cap N^+_{D'_1}(v'_i) = \emptyset.$$
 Thus $P''$ does not intersect with $S$. Since $b_{j_s}\in V_1$, Case 1 implies that there exists a path $P^*$ from $b_{j_s}$ to $V_2$ in $T[V\setminus S]$ with $E(P^* )\subseteq E_1$. Then $E(P')\cup E(P'')\cup E(P^*)$ contains a path $P$ in $T-S$ from $v$ to $V_2$ with $E(P)\subseteq E_0\cup E_1$. Thus we have $({\rm E}1)_2$. We can prove $({\rm E}1)_3$ in a similar way. This proves Claim~\ref{E1}. 
\end{proof}

\begin{claim}\label{E2}
There exist a set of arcs $E_2 \subseteq E(T)$ and two sets $W_2^+, W_2^- \subseteq V_2$ satisfying the following.
\begin{enumerate}
\item[$({\rm E}2)_1$] $|E_2|\leq k|V_2| - k $ and $  |W_2^+|,|W_2^-| \leq 2k-1$.
\item[$({\rm E}2)_2$] For a set $S\subseteq V(T)$ of size $k-1$ and a vertex $v\in V_2\setminus S$, there exists a path $P$ in $T-S$ from $v$ to $W_2^+$ with $E(P)\subseteq E_2.$
\item[$({\rm E}2)_3$] For a set $S\subseteq V(T)$ of size $k-1$ and a vertex $v\in V_2\setminus S$, there exists a path $P$ in $T-S$ from $W_2^-$ to $v$ with $E(P)\subseteq E_2.$
\end{enumerate}
\end{claim}
\begin{proof}[Proof of Claim~\ref{E2}]
We apply Lemma~\ref{main lemma} to $T[V_2]$ with parameters $0,k$ corresponding to $s,k$, respectively. Then we obtain an ordering $\sigma_2$ of $V_2$ and a $(\sigma_2,k,2k-1)$-good digraph $D_2\subseteq T[V_2]$ such that $|E(D_2)|\leq k|V_2|-k$. Let 
$$E_2:= E(D_2), \enspace W_2^-:= \sigma_1(1,2k-1) ~\text{ and }~ W_2^+:= \sigma_1(|V_2|-2k+2,|V_2|),$$ then we have $|E_2| = |E(D_2)| \leq k|V_2| - k$ and $|W_2^-|,|W_2^+| \leq  2k-1$. Hence we have (E2)$_1$. Since $D_2$ is $(\sigma_2,k,2k-1)$-good, Claim \ref{connectivity claim} implies that for any set $S$ of $k-1$ vertices in $V$ and a vertex $v\in V_2 \setminus S$, we can find a path $P$ in $T-S$ from $v$ to $W_2^+$ and a path $P'$ in $T-S$ from $W_2^-$ to $v$ such that $E(P), E(P')\subseteq E_2$, proving (E2)$_2$ and (E2)$_3$.
\end{proof}

Now we define $V_3, V_4$ as follows.
\begin{align}\label{V3 def}
V_3:= \bigcup_{i=1}^{k} V^{\rm int}(P_i) \setminus (V'_1\cup V_2) \enspace \text{ and } \enspace V_4:= V\setminus(V_1\cup V'_1 \cup V_2\cup V_3).
\end{align}

\begin{claim}\label{E3}
There exist a set of arcs $E_3 \subseteq E(T)$ and two sets $W_3^+, W_3^- \subseteq V_3$ satisfying the following.
\begin{enumerate}[{\rm (i)}]
\item[$({\rm E}3)_1$]  $|E_3|\leq (k-1)|V_3| + (k-1) $ and $|W_3^+|,|W_3^-| \leq 2k-1$.
\item[$({\rm E}3)_2$]   For a set $S\subseteq V(T)$ of size $k-1$ and a vertex $v\in V_3\setminus S$, there exists a path $P$ in $T-S$ from $v$ to $W_3^+\cup V_1$ with $E(P)\subseteq E_0\cup E_3.$
\item[$({\rm E}3)_3$]   For a set $S\subseteq V(T)$ of size $k-1$ and a vertex $v\in V_3\setminus S$, there exists a path $P$ in $T-S$  from $W_3^-\cup V_1$ to $v$ with $E(P)\subseteq E_0\cup E_3.$
\end{enumerate}
\end{claim}
\begin{proof}[Proof of Claim~\ref{E3}]
Consider a digraph $T[V_3]- E_0$. Note that $\delta(T[V_3]- E_0) \geq |V_3|-3$. Apply Lemma~\ref{main lemma} to $T[V_3]-E_0$ with parameters $2,k-1$ corresponding to $s,k$, respectively. Then we obtain an ordering $\sigma_3 = (v_1,\dots, v_{|V_3|})$ and a $(\sigma_3,k-1,2k-1)$-good digraph $D_3\subseteq T[V_3]-E_0$ with $|E(D_3)| \leq (k-1)|V_3|+ (k-1)$.
Here, it is important to take $(\sigma_3, k-1, 2k-1)$-good subgraph of $T[V_3] - E_0$ instead of $(\sigma_3, k, 2k-1)$-good subgraph of $T[V_3]$, otherwise we would get $|E(D_3)|\leq k|V_3|-k$ instead of $({\rm E}3)_1$.

Let $$E_3:= E(D_3), \enspace W_3^-:= \sigma_3(1,2k-1)~ \text{ and }~ W_3^+:= \sigma_3(|V_3|-2k+2,|V_3|).$$ 
This verifies $({\rm E}3)_1$. To verify $({\rm E}3)_2$, we consider a set $S\subseteq V(T)$ with $k-1$ vertices and a vertex $v\in V_3 \setminus S$. Then we consider a path $P'$ in $D_3-S$ with $E(P')\subseteq E(D_3)$ from $v$ to $v_i$ which maximizes $i$. If $i\geq |V_3|-2k+2$, then $v_i \in W^+_3$ and we are done. If $i < |V_3|-2k+2$, the maximality of $i$ implies $N_{D_3}^+(v_i)\subseteq S$ by (D1) and the fact that $D_3$ is $(\sigma,k-1,2k-1)$-good.  Since $$k-1 \stackrel{({\rm D}2)}{\leq} |N^+_{D_3}(v_i)| \leq |S| =k-1,$$ we have $S=N^+_{D_3}(v_i)$. 
Because $v_i \in V_3$, by \eqref{V3 def} there exists $s\in [3k-1]$ such that $v_i \in V^{\rm int}(P_s)$. We let $P''$ be the sub-path of $P_s$ from $v_i$ to $b_{j_s}$. Since $P_s$ is backwards-transitive, every vertex in $V(P'')$ should be in $N^-_T(v_i)$ except $v_i$ and the second vertex, say $u'$, of $P''$.
Since $\overrightarrow{v_iu'}\in E_0$ and $E(D_3)\subseteq T[V_3]-E_0$, $u'\notin N^+_{D_3}(v_i)$. Thus 
$$V(P'')\cap S \subseteq  (N^-_{T}(v_i)\cup \{v_i , u'\} ) \cap N^+_{D_3}(v_i) =\emptyset.$$ 
Thus $P''$ does not intersect with $S$. So $E(P')\cup E(P'')$ contains a path $P$ in $T-S$ from $v$ to $V_1$ with $E(P)\subseteq E_0 \cup E_3$. This proves $({\rm E}3)_2$. We can prove $({\rm E}3)_3$ in a similar way. This proves Claim~\ref{E3}.
\end{proof}

\begin{claim}\label{E4}
There exist a set of arcs $E_4 \subseteq A(T)$ and two sets $W_4^+, W_4^- \subseteq V_4$ satisfying the following.
\begin{enumerate}
\item[$({\rm E}4)_1$] $|E_4|\leq k|V_4| - k $ and $  |W_4^+|,|W_4^-| \leq 2k-1$.
\item[$({\rm E}4)_2$] For a set $S\subseteq V(T)$ of size $k-1$ and a vertex $v\in V_4\setminus S$, there exists a path $P$ in $T-S$ from $v$ to $W_4^+$ with $E(P)\subseteq E_4$.
\item[$({\rm E}4)_3$] For a set $S\subseteq V(T)$ of size $k-1$ and a vertex $v\in V_4\setminus S$, there exists a path $P$ in $T-S$ from $W_4^-$ to $v$ with $E(P)\subseteq E_4$.
\end{enumerate}
\end{claim}
\begin{proof}[Proof of Claim~\ref{E4}]
We apply Lemma~\ref{main lemma} to $T[V_4]$ with parameters $0,k$ corresponding to $s,k$, respectively. Then we obtain an ordering $\sigma_4$ and a $(\sigma_4,k,2k-1)$-good digraph $D_4\subseteq T[V_4]$ with $|E(D_4)|\leq k|V_4| - k$. 
Let $$E_4:= E(D_4),\enspace W_4^+:= \sigma_4(|V_4|-2k+2,|V_4|) \text{ and } W_4^-:= \sigma_4(1,2k-1),$$ then we have $|E_4| = |E(D_4)| \leq k|V_4| - k$, $|W_4^-| \leq 2k-1$ and $|W_4^+| \leq 2k-1$. Hence $({\rm E}4)_1$ holds. 
By Claim~\ref{connectivity claim}, for any $S \subseteq V(T)$ of $k-1$ vertices and $v\in V_4 \setminus S$, we can find a path $P$ in $T[V_4]\setminus S$ from $v$ to $W_4^+$ and a path $P'$ in $T[V_4]\setminus S$ from $W_4^-$ to $v$. This proves $({\rm E}4)_2$ and $({\rm E}4)_3$.  This proves Claim~\ref{E4}.
\end{proof}

We define $W^+$ and $W^-$ as follows. 
$$W^+:=W_2^+ \cup W_3^+ \cup W_4^+ \enspace \text{ and }\enspace W^- := W_2^- \cup W_3^- \cup W_4^-.$$ 
Note that $W^+, W^- \subseteq V\setminus (A\cup B)$. Thus $A$ in-dominates $W^+$ and $B$ out-dominates $W^-$. Now we take $E_5$ as follows to make connections from $W^+$ to $\{a_{i_1},\dots, a_{i_k}\}$ and from $\{b_{j_1},\dots, b_{j_k}\}$ to $W^-$. 

\begin{claim}\label{E5}
There exists a set of arcs $E_5\subseteq E(T)$ satisfying the following.
\begin{enumerate}[{\rm (i)}]
\item[$({\rm E}5)_1$]  $|E_5|\leq 81k^2$
\item[$({\rm E}5)_2$]  For $t\in [k]$, a vertex $v\in W^+$ and a set $S \subseteq V(T)\setminus \{a_{i_t},v\}$ of at most $k-1$ vertices, there exists a path $P(v,t)$ in $T-S$ from $v$ to $a_{i_t}$ such that $E(P(v,t))\subseteq E_5$.
\item[$({\rm E}5)_3$]  For $t\in [k]$, a vertex $v\in W^-$ and a set $S \subseteq V(T)\setminus \{b_{j_t},v\}$ of at most $k-1$ vertices, there exists a path $Q(v,t)$ in $T-S$ from $b_{j_t}$ to $v$ such that $E(Q(v,t))\subseteq E_5$.
\end{enumerate}
\end{claim}
\begin{proof}[Proof of Claim~\ref{E5}]
 By (A2) and (A3), for each $u\in W^+$ and $s\in [3k-1]$ there exists $c_{u,s} \in N^+_T(u) \cap A_s$ such that $c_{u,s}=a_{s}$ or $a_{s}\in N^+_T(c_{u,s})$. Let
$$P(u,s):= \left\{ \begin{array}{ll} (u,c_{u,s},a_s) & \text{ if } c_{u,s}\neq a_s, \\
(u,a_s) &\text{ otherwise.} \end{array}\right.
$$
Similarly, for $u\in W^-$ and $s\in [3k-1]$, there is a path $Q(u,s)$ from $b_s$ to $u$ with length at most $2$ lying entirely in $B_s\cup \{u\}$.
Let $$ E_5 := E(T[A_{\rm sink}]) \cup E(T[B_{\rm source}]) \cup \bigcup_{u\in W^+}\bigcup_{s=1}^{3k-1} E(P(u,s)) \cup \bigcup_{u\in W^-}\bigcup_{s=1}^{3k-1} E(Q(u,s)).$$
Then we have
\begin{align*}
|E_5| &\leq |E(T[A_{\rm sink}])| + |E(T[B_{\rm source}])|+ \sum_{u\in W^+}\sum_{s=1}^{3k-1} |E(P(u,s))| + \sum_{u\in W^-}\sum_{s=1}^{3k-1} |E(Q(u,s))|  \\
&\leq {{3k-1}\choose {2}} + {{3k-1}\choose{2}} + (6k-2)|W^+| + (6k-2)|W^-| \leq 81k^2.
\end{align*}

We get the final inequality from $({\rm E}2)_1$, $({\rm E}3)_1$ and $({\rm E}4)_1$. To verify $({\rm E}5)_2$, consider a set $S$ of $k-1$ vertices and an index $t\in [k]$ such 	that $a_{i_t}\notin S$ and a vertex $v\in W^+\setminus S$. Recall that $a_{i_t}$ has at least $k$ in-neighbors in $A_{\rm sink}$ as defined before Claim~\ref{in-AB}. This together with the fact that $A_1,\dots, A_{3k-1}$ are pairwise disjoint implies that there exists an index $s\in [3k-1]$ such that $a_s \in N^-_{T}(a_{i_t})$ and $A_s\cap S=\emptyset$. Then $P(v,s)\cup \overrightarrow{a_s a_{i_t}}$ contains a path $P$ from $v$ to $a_{i_t}$, where $P$ does not intersect with $S$ because $P$ is contained in $A_{s}\cup \{v\} \cup \{a_{i_t}\}$. Also $E(P)\subseteq E_5$, this proves $({\rm E}5)_2$. We can also prove $({\rm E}5)_3$ similarly. This proves Claim~\ref{E5}.
\end{proof}

Now we define the desired spanning strongly $k$-connected digraph $D\subseteq T$. Let 
$$V(D):=V(T)\enspace \text{ and }\enspace E(D):=E_0\cup E_1\cup E_2\cup E_3\cup E_4\cup E_5.$$ 
Because $\bigcup_{s=1}^{k} V^{\rm int}(P_s) \subseteq V'_1\cup V_2\cup V_3$, we have $|E_0| \leq |V'_1|+|V_2|+|V_3|-k$.
By $({\rm E}1)_1$, $({\rm E}2)_1$, $({\rm E}3)_1$, $({\rm E}4)_1$ and  $({\rm E}5)_1$ we have
\begin{eqnarray*}
|E(D)| &\leq& |E_0|+ |E_1|+|E_2|+|E_3|+|E_4|+|E_5| \\
 &\leq &(|V'_1|+|V_2|+|V_3|-k) +(k|V_1|+(k-1)|V'_1| +  680 k^2\log(k+1)) + (k|V_2|-k)  \\ 
&  &  + ((k-1)|V_3|+(k-1)) + (k|V_4|-k) + 81k^2 \\
&\leq& k(|V_1|+|V'_1|+|V_2|+|V_3|+|V_4|) + |V_2| + 740 k^2 \log(k+1) \\
& \stackrel{({\rm E}1)_1}{\leq}& k|V| + 750 k^2\log(k+1)
\end{eqnarray*}
since $680k^2 \log(k+1) + 81k^2 \leq 740k^2 \log (k+1)$ for $k \geq 2$.

Now it suffices to show that $D$ is strongly $k$-connected. For any set $S \subseteq V(T)$ of $k-1$ vertices and any two distinct vertices $u,v \in V(T)\setminus S$, we claim that there is a path from $u$ to $v$ in $D-S$. First of all, since $P_1,\dots, P_k$ are vertex-disjoint there exists $t \in [k]$ such that $V(P_t) \cap S = \emptyset$. We find a path $P$ in $D-S$ from $u$ to $u' \in W^+$ as follows.\newline

\noindent {\bf Case 1}. $u \in V_2\cup V_4$. \newline
There exists a path $P$ in $D-S$ from $u$ to $u' \in W^+$ by $({\rm E}2)_2$ and $({\rm E}4)_2$. \newline

\noindent {\bf Case 2}. $u \in V_1\cup V'_1$.\newline
By $({\rm E}1)_2$, there is a path $Q$ in $D-S$ from $u$ to a vertex $u_0\in V_2$. Also $({\rm E}2)_2$ implies that there is a path $Q'$ in $D-S$ from $u_0$ to $u' \in W^+$. Thus $E(Q)\cup E(Q')$ contains a path $P$ in $D-S$ from $u$ to $u' \in W^+$. \newline

\noindent {\bf Case 3}. $u\in V_3$.\newline
By $({\rm E}3)_2$, there is a path $R$ in $D-S$ from $u$ to a vertex $u_0 \in W^+ \cup V_1$. If $u_0 \in W^+$, then let $u' = u_0$ and $P:=R$. Otherwise, there is a path $R'$ in $D-S$ from $u_0$ to $u' \in W^+$ by Case 2. Thus $E(R)\cup E(R')$ contains a path $P$ in $D-S$ from $u$ to $u' \in W^+$. \newline

Similarly, there is a path $Q$ in $D-S$ from a vertex $v'\in W^-$ to $v$.  By Claim~\ref{E5}, there is a path $P(u',t)$ in $D-S$ from $u'$ to $a_{i_t}$, and a path $Q(v',t)$ in $D-S$ from $b_{j_t}$ to $v'$. Thus $E(P)\cup E(P(u',t))\cup E(P_t)\cup E(Q(v',t))\cup E(Q)$ contains a path in $D-S$ from $u$ to $v$. This proves that $D$ is strongly $k$-connected. \qed

\subsection*{Algorithmic aspect of Theorem~\ref{main theorem}}
%One may ask the following question that there is a polynomial-time algorithm to find a strongly $k$-connected spanning subgraph $D$ of at most $kn + 750k^2 \log(k+1)$ arcs, given an integer $k$ and a strongly $k$-connected $n$-vertex tournament $T$. 

The proof of Theorem~\ref{main theorem} is trivially algorithmic up to the following three optimization problems: finding a $k$-fan from a fixed vertex to a set with minimum total length, finding a maximum matching in a bipartite graph, and finding $k$ vertex-disjoint paths between two sets with minimum total length. These optimization problems can be solved in polynomial-time on $n=|V(T)|$ by standard application of algorithms finding maximum-flows and minimum cost flows of digraphs (see~\cite[Chapter 7,8 and 9]{ahuja}). Note that when we apply Lemma~\ref{main lemma}, we use Claim~\ref{claim2} to find the ordering $\sigma$ and a subgraph $D$ in polynomial time on $n$.
With these tools, the proof itself immediately gives a polynomial-time algorithm to find the desired digraph $D$ as in Theorem~\ref{main theorem}.

%\section{Concluding Remarks}\label{sec:conclude}
%\subsection*{Claims~\ref{E1} and~\ref{E3}}
%In the proof of Claim~\ref{E1} in Section 5, one may suggest to find a $(\sigma_1', k, 2k-1)$-good digraph of $T[V_1']$ rather than a $(\sigma_1', k-1, 2k-1)$-good digraph of $T[V_1'] - E_0$, which results in a better and shorter proof of Claim~\ref{E1}. The major difference between these two approaches is the upper bound of $|E_1|$. If we find a $(\sigma_1', k, 2k-1)$-good digraph of $T[V_1']$, this contributes $k|V_1'|$ to the upper bound of $|E_1|$, rather than $(k-1)|V_1'|$ of $({\rm E}1)_1$ of Claim~\ref{E1}. This yields that $|E(D)|$ is bounded above by $(k+1)n + g(k)$, which is worse than the upper bound in Theorem~\ref{main theorem} for sufficiently large $n$. It is also necessary to reduce $k|V_3|$ to $(k-1)|V_3|$ as in the statement of Claim~\ref{E3}, in order to prove the upper bound in Theorem~\ref{main theorem}.

\section{Acknowledgement}

We are grateful to Deryk Osthus for a careful reading and helpful comments. We thank the referees for a thorough reading and valuable suggestions.

\end{document}